\documentclass[aos,preprint]{imsart}

\RequirePackage{amsthm,amsmath,amsfonts,amssymb}
\RequirePackage[numbers,sort&compress]{natbib}
\RequirePackage[pass]{geometry} \RequirePackage{microtype} \RequirePackage[colorlinks,citecolor=blue,urlcolor=blue]{hyperref} \RequirePackage{graphicx} \RequirePackage{mathtools,bm,booktabs,enumitem,mathrsfs} \RequirePackage{placeins} \RequirePackage{tikz} \RequirePackage{pgfplots} \usepgfplotslibrary{groupplots}
\pgfplotsset{compat=1.18}
\startlocaldefs \theoremstyle{plain} \newtheorem{theorem}{Theorem}  \newtheorem{proposition}[theorem]{Proposition} \newtheorem{corollary}{Corollary} \theoremstyle{definition} \newtheorem{assumption}{Assumption} \newtheorem{example}{Example} \newtheorem{remark}{Remark} \newtheorem{algorithm}{Algorithm}

\newcommand{\R}{\mathbb R}
\newcommand{\E}{\mathbb E}
\newcommand{\Pp}{\mathbb P}
\newcommand{\diag}{\operatorname{diag}}
\newcommand{\supp}{\operatorname{supp}}
\newcommand{\dd}{\mathrm d}
\newcommand{\calF}{\mathcal F}
\newcommand{\calK}{\mathcal K}
\newcommand{\calP}{\mathcal P}
\newcommand{\bnorm}[1]{\left\lVert #1\right\rVert}

\newcommand{\veeop}{\mathbin{\vee}}
\newcommand{\ind}{\mathbf 1}

\endlocaldefs

\begin{document}
\microtypesetup{activate=false}

\begin{frontmatter}

\title{Ultra-high-dimensional spot support screening in continuous-time regression}
\runtitle{Spot support screening}

\begin{aug}

\author[A]{\fnms{Haibin}~\snm{Zhu}
\ead[label=e1]{haibinzhu@jnu.edu.cn}}

\address[A]{Department of Statistics and Data Science, School of Economics, Jinan University
\printead[presep={,\ }]{e1}}

\end{aug}

\begin{abstract}
We study variable screening for high-frequency continuous-time regression, targeting the support of the spot regression coefficient. Under coordinatewise truncation and finite-variation jumps, uniform entrywise bounds for local correlation estimates yield the rate $\{(\log p)/n\}^{\alpha/(2\alpha+1)}$ under a local averaged $\alpha$-H\"older condition, up to confidence and grid factors. These bounds imply sure screening for the marginal support. To retain the regression support under local dependence, we propose partial-residual spot iterative screening with pathwise Lepski adaptation. Under sparse eigenvalues and standardized beta-min, the procedure has the sure screening property and requires only correlation columns associated with the current active set. A list-size minimax lower bound applies to both support targets and matches the localization and multiplicity order of the marginal-screening bound when $\log(p/d)\asymp\log p$. Simulations and a high-frequency factor application illustrate the procedures.

\end{abstract}

\begin{keyword}[class=MSC]
\kwd{Primary 62G08, 62J05. Secondary 62G20, 60G44}
\end{keyword}

\begin{keyword}
\kwd{continuous-time regression}
\kwd{high-frequency data}
\kwd{variable screening}
\kwd{spot correlation}
\kwd{Lepski adaptation}
\end{keyword}
\end{frontmatter}

% ======================================================================
\section{Introduction}
\label{sec:introduction}

High-frequency observations make it possible to study regression relationships at a fixed point in time. The local nature of the problem, however, creates a distinct dimensionality constraint. A spot estimator uses only the increments in a shrinking neighborhood of the target time, so the relevant comparison is between the ambient dimension $p_n$ and the number of increments available in a local window, rather than the full-sample size $n$. This issue arises naturally in intraday panels containing large collections of assets, factors, and portfolio returns. Recent empirical work uses large panels of intraday factor and portfolio returns to study return predictability and high-frequency asset-pricing relations \cite{AletiBollerslevSiggaard2025,AletiBollerslev2025}. We study which coordinates of a continuous-time regression are active at a fixed time and seek a screening procedure that retains the spot regression support before a lower-dimensional model is fitted.

A large literature provides the statistical foundation for local estimation from high-frequency data. Realized covariation gives estimators of covariance, regression, and correlation for discretely observed semimartingales \cite{BarndorffNielsenShephard2004}, while local-constancy arguments provide a general approach to high-frequency inference \cite{MyklandZhang2009}. Further methods for uncertainty assessment under irregular and noisy high-frequency observations are developed in \cite{MyklandZhang2017}. Spot covariance and related volatility functionals have been studied under increasingly general high-frequency settings \cite{BibingerEtAl2019,JacodRosenbaum2013,BibingerHautschMalecReiss2014,BandiReno2018}. Tests for a diffusion component against pure-jump alternatives are developed in \cite{JingKongLiu2012PureJump}. Continuous-time regression and time-varying factor loadings are treated in \cite{AitSahaliaKalninaXiu2020,LiTodorovTauchen2017,LiTodorovTauchen2017Jump}, and uniform inference for spot volatility and beta is developed in \cite{JacodLiLiao2021}. For large panels, local and high-dimensional factor methods have been developed to accommodate time-varying factor structure \cite{Kong2017,KongLiu2018,KongLiuZhou2019,KongLinLiuLiu2023}. Robust factor estimation without moment constraints under elliptical distributions is studied in \cite{HeKongYuZhang2022}. High-dimensional volatility and covariance structure has also been studied using regularization and spectral methods \cite{AitSahaliaXiu2019,Kong2018,KimKongLiWang2018,Chen2024PCA}. Recent work studies realized regression with increasing dimension, asynchronicity, and microstructure noise \cite{ChenMyklandZhang2024}, as well as testing high-dimensional time-varying regression coefficients \cite{ChenFengMyklandZhang2026}. Local coefficient estimation is available through time-localized Dantzig estimation \cite{KimOhShin2026} and robust local Lasso estimation \cite{ShinKim2025}. We study retention of the spot regression support in a candidate list whose size is small relative to the local sample size.

Screening and penalized estimation serve complementary purposes. Screening aims to retain the active variables in a smaller candidate set, allowing inactive variables to remain. Methods such as Lasso jointly fit coefficients and select variables through regularization. Fan and Lv \cite{FanLv2008} develop sure independence screening as dimension reduction before procedures including Lasso and adaptive Lasso. Wang \cite{Wang2009} studies a corresponding role for forward screening. Lasso can itself serve as a screening step, as in the screen-and-clean framework of Wasserman and Roeder \cite{WassermanRoeder2009}. Here the objective is support retention under a prescribed list size, with the dimension compared to the number of local increments rather than the full sample size. Generalized, nonparametric, and robust screening criteria have also been developed \cite{FanSong2010,FanFengSong2011,LiPengZhangZhu2012,ChangTangWu2013}. Conditional procedures address limitations of marginal ranking in correlated designs \cite{BarutFanVerhasselt2016}, and forward selection has been studied for sparse varying coefficient models \cite{ChengHondaZhang2016}. 

The screening scores must be estimated from a local window that balances sampling variation and localization bias. Unknown local smoothness and jumps require simultaneous control over the growing coordinate panel. A further difficulty is that marginal and regression relevance need not coincide. After standardization,
\[
r_\tau=R_\tau\theta_\tau,
\]
where $R_\tau$ is the spot covariate correlation matrix and $\theta_\tau$ is the standardized spot regression coefficient. Local dependence can hide active variables through marginal cancellation or make inactive variables marginally prominent. We therefore study marginal support screening and use partial residualization to target the regression support.

Our first contribution is a uniform theory for the local correlation quantities used by screening. Under coordinatewise truncation and finite-variation jumps, we establish simultaneous entrywise bounds over a grid of local windows. Lepski-type selection \cite{Lepski1991,GoldenshlugerLepski2011,PatschkowskiRohde2019,CleanthousGeorgiadisLepski2025} adapts the window to unknown local smoothness. Under a local averaged $\alpha$-H\"older condition, the rate is $\left\{\log p_n/n\right\}^{\alpha/(2\alpha+1)}$, up to confidence and grid factors, yielding sure screening for the marginal support.

Our second contribution concerns the regression support. We propose partial-residual spot iterative screening, abbreviated PR-SISIS. At each update, the procedure removes the linear contribution of the current active set and ranks the remaining variables by local partial-residual covariance scores. The window is reselected using the marginal vector and the correlation columns of the active set. Under sparse eigenvalues and standardized beta-min, uniform control of the residual scores yields sure screening for the regression support. With model cap $q_n$, the procedure computes at most $q_n$ correlation columns, involving $O(p_nq_n)$ distinct covariate pairs rather than a dense $p_n\times p_n$ matrix.

Finally, a list-size minimax lower bound applies to both support targets and to all measurable screening rules. When $\log(p_n/d_n)\asymp\log p_n$, it matches the localization and multiplicity order of the marginal-screening bound.

Simulation studies examine the effects of signal strength, dimension, dependence, and retained-list size. We also study window selection under changing local relationships at different sampling frequencies. The empirical illustration uses a high-frequency factor panel to compare marginal and partial-residual screening at a fixed intraday time, with window comparisons at additional intraday targets. The two procedures produce different retained sets in this dependent panel, which illustrates the distinction between marginal and conditional local relevance. 

The remainder of this paper is organized as follows. Section~\ref{sec:methods} introduces the model and screening procedures. Section~\ref{sec:theory} presents the theoretical results. Sections~\ref{sec:simulation} and \ref{sec:empirical} report the simulation and empirical results. Section~\ref{sec:conclusions} concludes. Additional simulation results and all proofs are provided in the Supplementary Material \cite{ZhuSupplement2026}. The proposed screening methods are available in the \textsf{R} package \texttt{hfsis} at \url{https://github.com/HBZhuLab/hfsis}. 

Throughout, $p=p_n$ denotes the ambient dimension, which may depend on $n$, and $\bar p_n=p_n+1$. For a vector $v\in\R^{p_n}$ and an index set $A\subseteq\{1,\ldots,p_n\}$, $v_A$ denotes the subvector indexed by $A$. For a matrix $M$, $M_{AB}$ denotes the submatrix with row indices in $A$ and column indices in $B$, and $M_{AA}$ is a principal submatrix. We write
\[
\supp(v)=\{j:v_j\neq0\},\qquad \bnorm{v}_q=\Big(\sum_j|v_j|^q\Big)^{1/q},
\]
with the usual modification for $q=\infty$, and
\[
\bnorm{M}_{\max}=\max_{j,k}|M_{jk}|,\qquad \bnorm{M}_{\mathrm{op}}=\sup_{\bnorm{v}_2=1}\bnorm{Mv}_2.
\]
For a symmetric matrix $M$, $\lambda_{\min}(M)$ denotes its smallest eigenvalue, $M\succeq0$ means that $M$ is positive semidefinite, and $I_d$ denotes the $d\times d$ identity matrix. The notation $a_n\lesssim b_n$ means $a_n\le Cb_n$ for a constant $C$ independent of $n$, and $a_n\asymp b_n$ means $a_n\lesssim b_n$ and $b_n\lesssim a_n$. All limits are taken as $n\to\infty$. The symbols $C,c,c_1,c_2,\ldots$ denote positive constants whose values may change from line to line.

\section{Model and screening procedures}
\label{sec:methods}

The marginal and regression supports need not coincide under local dependence. Their population relation leads to two screening procedures based on the same local correlation estimates.

\subsection{Model, spot targets, and identification}
\label{subsec:model-targets}

We observe a real-valued response process $Y$ and an $\R^{p_n}$-valued covariate process
\[
X=(X_1,\ldots,X_{p_n})^\top
\]
on $[0,1]$ at the regular times $\{t_i^n\}_{i=0,\ldots,n}$, with $t_i^n=i\Delta_n$ and $\Delta_n=n^{-1}$. For any c\`adl\`ag process $H$, define
\[
\Delta_i^nH=H_{t_i^n}-H_{t_{i-1}^n}.
\]
For a semimartingale $H$, $H^c$ denotes its continuous local martingale part. For two continuous local martingales $M$ and $N$, $\langle M,N\rangle$ denotes their predictable quadratic covariation. The superscript $(n)$ on the processes is suppressed when no confusion can arise. The object of interest is local at a fixed interior point $\tau\in(0,1)$.

Let $(\Omega,\calF,(\calF_t)_{t\in[0,1]},\Pp)$ satisfy the usual conditions. The regression model is
\begin{equation}
\dd Y_t=\beta_{t-}^\top\dd X_t+\dd\varepsilon_t, \label{eq:regression-model}
\end{equation}
where $\beta$ is adapted, c\`adl\`ag, sparse, and locally bounded on the estimation neighborhood. The left limit $\beta_{t-}$ is the predictable integrand, and at every jump time
\begin{equation*}
\Delta Y_t=\beta_{t-}^\top\Delta X_t+\Delta\varepsilon_t.
\end{equation*}
The residual process $\varepsilon$ is a real-valued It\^o semimartingale, and the $(p_n+1)$-dimensional process $(X^\top,\varepsilon)^\top$ is a joint It\^o semimartingale. We impose the strong orthogonality condition
\begin{equation*}
\langle X^c,\varepsilon^c\rangle_t=0, \qquad t\in[0,1],
\end{equation*}
where the left-hand side is understood componentwise.

Collect the observed processes in
\begin{equation*}
Z_t=(X_t^\top,Y_t)^\top\in\R^{\bar p_n}.
\end{equation*}
Possibly after passing to a very good filtered extension of the original space, take a common Grigelionis representation of the joint It\^o semimartingale $(X^\top,\varepsilon)^\top$ on a Polish mark space $(\mathcal E,\mathscr E)$, driven by a $d_{W,n}$-dimensional Brownian motion $W$ and a Poisson random measure $\mu$ on $[0,1]\times\mathcal E$ with compensator $\nu(\dd t,\dd z)=\dd t\,\lambda(\dd z)$, where $\lambda$ is a $\sigma$-finite measure. For a predictable function $f$, write
\[
f\star\mu_t=\int_0^t\int_{\mathcal E}f(s,z)\,\mu(\dd s,\dd z), \qquad f\star(\mu-\nu)_t=\int_0^t\int_{\mathcal E}f(s,z)\,(\mu-\nu)(\dd s,\dd z).
\]
The induced representation of $Z$ is
\begin{equation*}
\begin{aligned} Z_t={}&Z_0+\int_0^t b_s\,\dd s+\int_0^t\sigma_s\,\dd W_s\\ &+(\delta\ind_{\{\bnorm{\delta}_\infty\le1\}})\star(\mu-\nu)_t +(\delta\ind_{\{\bnorm{\delta}_\infty>1\}})\star\mu_t. \end{aligned}
\end{equation*}
where $b_t\in\R^{\bar p_n}$, $\sigma_t\in\R^{\bar p_n\times d_{W,n}}$, and $\delta(t,z)\in\R^{\bar p_n}$ are the induced drift, diffusion, and jump coefficients of $Z$ under the displayed truncation convention, respectively. See \cite{JacodProtter2012} for more details.

The continuous covariance density of $Z$ is
\begin{equation*}
\Sigma_t=\sigma_t\sigma_t^\top.
\end{equation*}
We assume that $\Sigma$ admits a c\`adl\`ag version on a right neighborhood of $\tau$ and that, almost surely, all diagonal entries of this version are strictly positive throughout that neighborhood. We fix this version throughout. Let $\delta^X(t,z)\in\R^{p_n}$ and $\delta^\varepsilon(t,z)\in\R$ denote the jump coefficients of $X$ and $\varepsilon$, respectively, in their joint representation. The jump coefficient of the response and the augmented jump coefficient are
\begin{equation*}
\delta^Y(t,z)=\beta_{t-}^\top\delta^X(t,z)+\delta^\varepsilon(t,z), \qquad \delta(t,z)=\bigl((\delta^X(t,z))^\top,\delta^Y(t,z)\bigr)^\top.
\end{equation*}

Let $C_t$, $c_t$, and $c_{YY,t}$ denote the corresponding blocks of this fixed c\`adl\`ag version. Define $c_{\varepsilon\varepsilon,t}=c_{YY,t}-\beta_t^\top C_t\beta_t$, which is a c\`adl\`ag version of $\dd\langle\varepsilon^c\rangle_t/\dd t$. Then
\begin{equation*}
C_t=\frac{\dd\langle X^c\rangle_t}{\dd t}, \qquad c_t=\frac{\dd\langle X^c,Y^c\rangle_t}{\dd t}=C_t\beta_t,
\end{equation*}
and
\begin{equation*}
c_{YY,t}=\frac{\dd\langle Y^c\rangle_t}{\dd t} =\beta_t^\top C_t\beta_t+c_{\varepsilon\varepsilon,t}.
\end{equation*}
The process $\beta_-$ is used in the stochastic integral, whereas $\beta$ is used in the c\`adl\`ag density relation because $\beta_- = \beta$ outside a $\dd t\otimes\dd\Pp$ null set. Strong orthogonality gives the displayed density identities $\dd t\otimes\dd\Pp$-almost everywhere, and since both sides have c\`adl\`ag versions, the identities extend indistinguishably to every time point. Thus the forward window beginning at $\tau$ targets the right-continuous coefficient $\beta_\tau$.

Define
\[
D_t=\diag(C_{11,t},\ldots,C_{p_n p_n,t}), \qquad R_t=D_t^{-1/2}C_tD_t^{-1/2}.
\]
The marginal spot correlation vector and the standardized regression coefficient are
\[
r_t=\frac{D_t^{-1/2}c_t}{\sqrt{c_{YY,t}}}, \qquad \theta_t=\frac{D_t^{1/2}\beta_t}{\sqrt{c_{YY,t}}},
\]
so that
\begin{equation}
r_t=R_t\theta_t. \label{eq:spot-identification}
\end{equation}
The augmented spot correlation matrix is
\begin{equation*}
\mathcal R_t= \begin{pmatrix} R_t&r_t\\ r_t^\top&1 \end{pmatrix}\succeq0,
\end{equation*}
which is positive semidefinite because it is the instantaneous correlation matrix of the martingale parts of $(X^\top,Y)^\top$.

The identity $r_t=R_t\theta_t$ leads to two different support targets:
\begin{equation*}
S_\tau^\beta=\supp(\beta_\tau)=\supp(\theta_\tau), \qquad S_\tau^m=\supp(r_\tau).
\end{equation*}
The first is the regression support and is our primary target. The second is the set visible to marginal correlation screening. They coincide when $R_\tau=I_{p_n}$, but local dependence can make them different.

The following example shows how local dependence can suppress the marginal signal of an active variable and inflate that of an inactive surrogate.

\begin{example}\label{ex:population}
Let the covariates have unit spot variances. The first six coordinates have three independent $2\times2$ correlation blocks with correlations $0.90$, $0.85$, and $0.80$, respectively. Take
\[
\beta_\tau=(0.50,-0.52,0.40,-0.35,0.30,0)^\top
\]
and set the spot residual variance to one. After standardizing $Y$, active coordinate 4 has absolute marginal score $0.009$, whereas inactive coordinate 6 has score $0.220$. Residualizing each coordinate against its block partner gives active scores between $0.087$ and $0.102$ and an inactive score of zero for coordinate 6. Thus marginal screening targets $S_\tau^m$, while partial-residual screening can recover signals in $S_\tau^\beta$ that are hidden by local dependence.
\end{example}

\subsection{Local estimation and screening procedures}
\label{subsec:procedures}

The spot quantities in \eqref{eq:spot-identification} are not observed. They must be estimated over a window containing $k$ increments and spanning $h=k/n$ units of calendar time. Under local $\alpha$-H\"older regularity, the leading entrywise error has the form
\begin{equation*}
\underbrace{\sqrt{\frac{\log p}{k}}}_{\text{sampling variation}} + \underbrace{\left(\frac{k}{n}\right)^\alpha}_{\text{localization bias}},
\end{equation*}
where $\alpha$ controls the smoothness of local covariance paths of $\Sigma_t$. The first term decreases with the window size and the second increases. Their balance gives the oracle window order and error rate

\[
k_n^\star\asymp n^{2\alpha/(2\alpha+1)}(\log p)^{1/(2\alpha+1)},\qquad \varepsilon_n\asymp\left\{\frac{\log p}{n}\right\}^{\alpha/(2\alpha+1)}.
\]

The formal analysis below includes the Bernstein linear term, the bandwidth grid, and the confidence level. The marginal procedure uses only the estimated correlation vector. At each partial-residual update, the bandwidth comparison augments that vector by the correlation columns of the current active set, so the comparison object changes with the active set.

For an integer $k$, let

\[
\tau_n=\Delta_n\lceil n\tau\rceil,\qquad I_{n,k}(\tau)=\{\lceil n\tau\rceil+1,\ldots,\lceil n\tau\rceil+k\}.
\]

This forward window contains $k$ increments and targets the right-continuous coefficient. A symmetric window also targets this coefficient when the covariance density is continuous at $\tau$. Since $\tau\in(0,1)$ and $k_{\max}/n\to0$, we have $\tau_n+k_{\max}\Delta_n\le1$ for all sufficiently large $n$.

Fix constants $0<\alpha_0<\infty$ and $0<\underline s\le\overline s<\infty$, and let $s_{a,n}$, $a=1,\ldots,\bar p_n$, be possibly random coordinate scales. We assume that there are events $\mathcal H_n$ such that
\begin{equation}
\Pp(\mathcal H_n)\to1, \qquad \underline s\le\min_{a\le\bar p_n}s_{a,n}\le\max_{a\le\bar p_n}s_{a,n}\le\overline s \quad\text{on }\mathcal H_n. \label{eq:scale-bounds}
\end{equation}
No independence between the scales and the observed increments is required. Let $r_0\in[0,1)$ be the jump-activity index in Assumption~\ref{ass:jump}, and choose $\varpi$ satisfying
\begin{equation}
\frac{1}{2(2-r_0)}<\varpi<\frac12. \label{eq:threshold-range}
\end{equation}
The parameter $r_0$ controls small-jump activity in the usual high-frequency sense. See, for example, \cite{AitSahaliaJacod2009,JingKongLiuMykland2012Activity}. The range in \eqref{eq:threshold-range} is the standard finite-variation truncation regime. The uniform high-dimensional requirements used below are imposed separately in \eqref{eq:grid-compatibility}.

Set $\vartheta_n=\Delta_n^\varpi$ and $\vartheta_{a,n}=\alpha_0s_{a,n}\vartheta_n$. By \eqref{eq:scale-bounds}, these thresholds are uniformly comparable to $\vartheta_n$ on $\mathcal H_n$. For each coordinate, define the truncated increment
\begin{equation*}
\widetilde\Delta_i^n Z_a =\Delta_i^nZ_a \ind_{\{|\Delta_i^nZ_a|\le \vartheta_{a,n}\}}, \qquad \widetilde\Delta_i^nZ =(\widetilde\Delta_i^nZ_1,\ldots, \widetilde\Delta_i^nZ_{\bar p_n})^\top,
\end{equation*}
Coordinatewise truncation accommodates heterogeneous scales without a vector norm that grows with dimension. The theory allows random scales satisfying \eqref{eq:scale-bounds}. In the numerical study, we use a full-sample bipower scale and threshold exponent 0.47. Section~\ref{sec:simulation} states the rule, and the Supplement gives further implementation details.

Denote the local truncated realized covariance matrix as
\begin{equation*}
\widehat\Sigma^{\mathrm{tr}}(k) =\frac{1}{k\Delta_n} \sum_{i\in I_{n,k}(\tau)} \widetilde\Delta_i^nZ (\widetilde\Delta_i^nZ)^\top.
\end{equation*}
Because truncation is applied to the vector before the outer product is formed, $\widehat\Sigma^{\mathrm{tr}}(k)$ is positive semidefinite for every $k$.

Let $v_n\downarrow0$ and set
\[
\widehat V(k)=\diag\{\widehat\Sigma_{11}^{\mathrm{tr}}(k)\vee v_n,\ldots,\widehat\Sigma_{\bar p_n\bar p_n}^{\mathrm{tr}}(k)\vee v_n\}.
\]
\begin{remark}
The deterministic floor $v_n\downarrow0$ keeps correlation normalization well defined when a local variance estimate is zero or numerically small. On the common event in Theorem~\ref{thm:joint-adaptation}, all covariance diagonals eventually exceed $c_-/2$, so the floor is inactive and does not affect the rate.
\end{remark}

The estimated augmented correlation matrix is
\begin{equation}
\widehat{\mathcal R}(k) =\widehat V(k)^{-1/2}\widehat\Sigma^{\mathrm{tr}}(k)\widehat V(k)^{-1/2} = \begin{pmatrix} \widehat R(k)&\widehat r(k)\\ \widehat r(k)^\top&\widehat R_{YY}(k) \end{pmatrix}. \label{eq:estimated-augmented-correlation}
\end{equation}
Correlation normalization is essential. Raw covariance ranking favors volatile covariates, while local slope ranking favors low-volatility covariates. Rescaling each coordinate and its truncation threshold by the same positive factor leaves $\widehat r(k)$ unchanged when the variance floor is inactive.

We choose the window from a geometric grid $\calK_n=\{k_{n,1}<\cdots<k_{n,M_n}\}$ with $M_n=|\calK_n|$, $k_{n,1}=k_{\min}\to\infty$, $k_{n,M_n}=k_{\max}$, and $k_{\max}/n\to0$. For fixed constants $1<\rho_-\le\rho_+<\infty$, assume
\[
\rho_-\le\frac{k_{n,m+1}}{k_{n,m}}\le\rho_+, \qquad m=1,\ldots,M_n-1.
\]
This gives $O(\log n)$ candidate windows and places an oracle window within a fixed factor of a grid point. A dyadic grid is the special case with both ratios equal to two.

The Lepski comparison requires a stochastic radius that controls all entries of the estimated correlation vector and matrix simultaneously over the candidate windows. Let \(x=x_n\geq0\) be a deterministic confidence index and define

\begin{equation*}
\Lambda_n(x) = 2\log(e\bar p_n) + \log\{2M_n\} + x, \qquad a_{n,k}(x) = \sqrt{\frac{\Lambda_n(x)}{k}} + \frac{\Lambda_n(x)}{k}.
\end{equation*}

The terms \(2\log(e\bar p_n)\) and \(\log\{2M_n\}\) account for the \(O(\bar p_n^2)\) covariance entries and the candidate windows, respectively. The confidence index \(x\) determines the exponential tail probability. In the fixed-window stochastic radius \(a_{n,k}(x)\), the square-root and linear terms are the variance and large-deviation terms of the Bernstein bound.

We impose the jump-truncation conditions

\begin{equation}
\begin{gathered} \log(e\bar p_n k_{\max}M_n)+x=o(n^{1-2\varpi}),\\ \frac{M_n^{1/2}\vartheta_n^{1-r_0/2} +M_n\vartheta_n^{2-r_0}} {a_{n,k_{\max}}(x)+n^{-1/2}}=o(1). \end{gathered} \label{eq:grid-compatibility}
\end{equation}

The first condition keeps continuous increments below the thresholds uniformly over coordinates and local increments. The second makes the truncation error negligible over the candidate grid. The smallest denominator occurs at $k_{\max}$.

In particular, up to the displayed grid and confidence factors, the first condition requires $\log p_n=o(n^{1-2\varpi})$. For an active set $A\subseteq\{1,\ldots,p_n\}$ and two candidate windows $k,\ell\in\calK_n$, let
\[
d_{k\ell}(A) =\bnorm{\widehat r(k)-\widehat r(\ell)}_\infty \veeop \bnorm{\widehat R_{\cdot A}(k)-\widehat R_{\cdot A}(\ell)}_{\max},
\]
with the second term defined as zero for $A=\varnothing$. The active-set Lepski selector is
\begin{equation}
\widehat k(A) =\max\left\{ k\in\calK_n: d_{k\ell}(A) \leq C_{\mathrm L}\{a_{n,k}(x)+a_{n,\ell}(x)\} \text{ for all }\ell\in\calK_n,\ \ell\leq k \right\}. \label{eq:lepski-selector}
\end{equation}
The smallest grid point is always admissible, so the selector is well defined. When $A=\varnothing$, \eqref{eq:lepski-selector} is a marginal-only comparison. As variables enter the active set, subsequent comparisons also use their correlation columns.

For a list size $1\le d\le p_n$, the marginal set is

\begin{equation*}
\widehat S_m(d) = \left\{ j:|\widehat r_j(\widehat k(\varnothing))| \text{ is among the largest $d$ scores} \right\},
\end{equation*}
where ties are resolved by increasing coordinate index, a fixed deterministic rule. This rule directly estimates the marginal support. To recover the regression support when active variables are hidden by local dependence, we next remove the linear contribution of the current model.

For an active set $A$ and $j\notin A$, the population partial-residual covariance score is
\begin{equation*}
\zeta_{j\mid A} =r_{j,\tau}-R_{jA,\tau}R_{AA,\tau}^{-1}r_{A,\tau},
\end{equation*}
and its sample analogue is
\begin{equation*}
\widehat\zeta_{j\mid A}(k) =\widehat r_j(k) -\widehat R_{jA}(k) \widehat R_{AA}(k)^{-1} \widehat r_A(k).
\end{equation*}
When $A=\varnothing$, set
\[
\zeta_{j\mid\varnothing}=r_{j,\tau}, \qquad \widehat\zeta_{j\mid\varnothing}(k)=\widehat r_j(k).
\]
The score is the covariance between the residual of standardized $X_j$ after projection on $X_A$ and the similarly residualized response. We use the covariance numerator rather than the fully normalized partial correlation because it is sufficient for sure screening and avoids requiring a uniform lower bound on every residual variance.

\begin{algorithm}[Pathwise joint-Lepski PR-SISIS]
\label{alg:pr-sisis}
Fix a model-size cap $1\le q_n\le p_n$ and a block size $m_n\ge1$, and set $L_n=\lceil q_n/m_n\rceil$ and $A^{(0)}=\varnothing$. For $s=0,\ldots,L_n-1$:
\begin{enumerate}[label=(\roman*)]
  \item Select $\widehat k_s=\widehat k(A^{(s)})$.
  \item If $A^{(s)}\ne\varnothing$ and $\widehat R_{A^{(s)}A^{(s)}}(\widehat k_s)$ is singular, stop and return $A^{(s)}$. Otherwise, compute $\widehat\zeta_{j\mid A^{(s)}}(\widehat k_s)$ for all $j\notin A^{(s)}$.
  \item Let $\widetilde m_s=\min\{m_n,q_n-|A^{(s)}|\}$ and add the $\widetilde m_s$ largest absolute scores, resolving ties by increasing coordinate index:
\[
A^{(s+1)} =A^{(s)}\cup \left\{ j:|\widehat\zeta_{j\mid A^{(s)}}(\widehat k_s)| \text{ is among the largest $\widetilde m_s$} \right\}.
\]
  \item Stop before the next iteration if $|A^{(s+1)}|=q_n$.
\end{enumerate}
Let $A^{\mathrm{fin}}$ be the active set at termination and set $\widehat S_{\mathrm{PR}}=A^{\mathrm{fin}}$. For the complete ordering used in the numerical summaries, concatenate the selected blocks in update order and order each block by decreasing absolute scores used at that update, resolving ties by increasing coordinate index. If a singular block caused termination, return only the retained set. If all coordinates have been selected, the concatenated blocks give the complete ordering without another window selection or solve. If unselected coordinates remain after normal completion, select $\widehat k_{\mathrm{fin}}=\widehat k(A^{\mathrm{fin}})$ and check $\widehat R_{A^{\mathrm{fin}}A^{\mathrm{fin}}}(\widehat k_{\mathrm{fin}})$ before computing the remaining scores. When this block is nonsingular, append the remaining coordinates in decreasing order of $|\widehat\zeta_{j\mid A^{\mathrm{fin}}}(\widehat k_{\mathrm{fin}})|$, with the same tie rule. If it is singular, return the retained set without a complete ordering.
\end{algorithm}

PR-SISIS retains variables that are marginally visible or become visible after local projections. The returned set can be used in a subsequent local Lasso or another reduced-model fit. The guarantee established below concerns retention of the spot support within the prescribed model-size cap.

\begin{remark}
With model cap $q_n$, at most $q_n$ correlation columns are computed and cached. The method therefore involves $O(p_nq_n)$ distinct covariate pairs rather than a dense $p_n\times p_n$ matrix. For fixed $q_n$ and a fixed window grid, a direct implementation is linear in $p_n$. Detailed arithmetic and memory bounds are given in the Supplement.
\end{remark}

\section{Theoretical results}
\label{sec:theory}

We first establish uniform correlation bounds and derive sure screening for the marginal support. The same bounds then yield sure screening for the regression support through partial residualization. We conclude with a minimax lower bound under a retained-list constraint.

Unless stated otherwise, all conditions below involving random population quantities evaluated at $\tau$ are understood to hold almost surely.

\subsection{Adaptive spot correlations and marginal screening}
\label{subsec:marginal-theory}

Let
\[
\Sigma_t =\frac{\dd\langle Z^c\rangle_t}{\dd t} = \begin{pmatrix} C_t&c_t\\ c_t^\top&c_{YY,t} \end{pmatrix}.
\]
To obtain uniform truncation bounds over the growing coordinate panel, we control the local drift, volatility, and finite-variation jump activity uniformly across coordinates.
\begin{assumption}
\label{ass:jump}
Let $I_n^{\max}=[\tau_n,\tau_n+k_{\max}\Delta_n]$, and
\begin{equation*}
b'_t=b_t-\int_{\{\bnorm{\delta(t,z)}_\infty\le1\}} \delta(t,z)\lambda(\dd z).
\end{equation*}
There exist a constant $K<\infty$, independent of $n$, time, and the coordinate index, and an activity index $r_0\in[0,1)$ such that almost surely on $I_n^{\max}$,
\begin{equation*}
\max_{a\le\bar p_n}\{|b_{a,t}|+|b'_{a,t}|\} +\max_{a\le\bar p_n}\Sigma_{aa,t}\le K,
\end{equation*}
and
\begin{equation*}
\int_{\mathcal E}\sup_{a\le\bar p_n} \left(|\delta_a(t,z)|^{r_0}\ind_{\{\delta_a(t,z)\ne0\}}\wedge1\right) \lambda(\dd z)\le K.
\end{equation*}
For $r_0=0$, the expression $|u|^{r_0}\ind_{\{u\ne0\}}$ is interpreted as $\ind_{\{u\ne0\}}$.
\end{assumption}

Conditions in Assumption \ref{ass:jump} control finite-variation jump activity over the local window. The next condition separates localization bias from stochastic error. It controls local covariance averages over every candidate window and keeps the corresponding marginal variances uniformly nondegenerate.
\begin{assumption}
\label{ass:local-characteristics}
There exist constants $\alpha\in(0,1]$, $0<L_\Sigma<\infty$, and $0<c_-<c_+<\infty$, and events $\mathcal G_n$ with $\Pp(\mathcal G_n)\to1$, such that on $\mathcal G_n$, for every $k\in\calK_n$,
\begin{equation*}
\max_{1\le a,b\le\bar p_n} \left| \frac{1}{k\Delta_n} \int_{\tau_n}^{\tau_n+k\Delta_n}\Sigma_{ab,s}\,\dd s -\Sigma_{ab,\tau} \right| \le L_\Sigma(k\Delta_n)^\alpha,
\end{equation*}
and
\begin{equation}
0<c_-\le\min_{a\le\bar p_n}\Sigma_{aa,t} \le\max_{a\le\bar p_n}\Sigma_{aa,t}\le c_+<\infty \qquad\text{on }I_n^{\max}. \label{eq:variance-bounds}
\end{equation}
\end{assumption}

After adjusting $c_-$ and $c_+$ by fixed factors, \eqref{eq:variance-bounds} also holds at $t=\tau$. Assumption~\ref{ass:local-characteristics} controls the local average of the covariance process. For diffusion-type covariance factors, any fixed $\alpha<1/2$ is natural under suitable local bounds, whereas smoother paths allow larger $\alpha$. A factor-type sufficient condition is given in Section~\ref{sec:supp-factor-condition} of the Supplement.

The proof separates localization bias, continuous sampling error, and truncation error. Define the adaptive rate

\[
\varepsilon_n(x)=\inf_{k\in\calK_n} \left\{a_{n,k}(x)+\left(\frac{k}{n}\right)^\alpha\right\} +n^{-1/2}.
\]

\begin{theorem}\label{thm:joint-adaptation}
For any deterministic sequence $x=x_n\ge0$ satisfying \eqref{eq:grid-compatibility}, suppose the scale condition \eqref{eq:scale-bounds}, Assumptions~\ref{ass:jump} and \ref{ass:local-characteristics}, and \eqref{eq:threshold-range} hold, $v_n<c_-/2$ for all sufficiently large $n$, and
\begin{equation}
a_{n,k_{\min}}(x) +\left(\frac{k_{\max}}{n}\right)^\alpha +n^{-1/2} \le c_0 \label{eq:grid-condition}
\end{equation}
for a sufficiently small constant $c_0$. Then there exist constants $C,c>0$ such that, with probability at least $1-Ce^{-cx}-o(1)$,
\begin{equation*}
\bnorm{\widehat r(k)-r_\tau}_\infty \veeop \bnorm{\widehat R(k)-R_\tau}_{\max} \le C\{a_{n,k}(x)+(k/n)^\alpha+n^{-1/2}\} \quad\text{for every }k\in\calK_n.
\end{equation*}
\end{theorem}

Because the event in Theorem~\ref{thm:joint-adaptation} controls every matrix entry, its restriction to the columns indexed by $A$ is valid for data-dependent sets. Applying the Lepski comparison on this event gives the following bound.
\begin{corollary}
\label{cor:pathwise-lepski}
Under the conditions of Theorem~\ref{thm:joint-adaptation}, suppose that $C_{\mathrm L}$ in \eqref{eq:lepski-selector} is larger than a fixed constant determined by the fixed-window bound and grid-comparison factors, and that the grid contains a point within a fixed factor of the bias and variance balance. Then, on the same event as Theorem~\ref{thm:joint-adaptation}, simultaneously for every $A\subseteq\{1,\ldots,p_n\}$ with $|A|\le q_n$,
\begin{equation*}
\bnorm{\widehat r\{\widehat k(A)\}-r_\tau}_\infty \veeop \bnorm{\widehat R_{\cdot A}\{\widehat k(A)\}-R_{\cdot A,\tau}}_{\max} \le C\varepsilon_n(x).
\end{equation*}
\end{corollary}

The selector and its stochastic radii do not use $\alpha$. For each fixed $\alpha$ satisfying Assumption~\ref{ass:local-characteristics}, the rate follows when the grid covers the corresponding bias and variance balance.

When $\Lambda_n(x)=o(n)$ and the grid is sufficiently rich, we have
\begin{equation*}
\varepsilon_n(x) \asymp \left\{\frac{\Lambda_n(x)}{n}\right\}^{\alpha/(2\alpha+1)}, \qquad k_n^\star \asymp n^{2\alpha/(2\alpha+1)}\Lambda_n(x)^{1/(2\alpha+1)},
\end{equation*}
up to the displayed $n^{-1/2}$ term. Under the grid conditions in \eqref{eq:grid-compatibility}, finite-variation jumps do not change this first-order boundary.

Applying Corollary~\ref{cor:pathwise-lepski} at $A=\varnothing$ gives the estimator used by marginal screening. Let

\[
a_m=\min_{j\in S_\tau^m}|r_{j,\tau}|,
\]

with $a_m=\infty$ when $S_\tau^m=\varnothing$. The following theorem gives sure screening for the marginal support.

\begin{theorem}\label{thm:marginal-screening}
Suppose the conditions of Corollary~\ref{cor:pathwise-lepski} hold, and let $1\le d_n\le p_n$ be an integer. If $a_m\ge C\varepsilon_n(x)$ and $d_n\ge|S_\tau^m|$, we have
\begin{equation}
\Pp\{S_\tau^m\subseteq\widehat S_m(d_n)\} \ge1-Ce^{-cx}-o(1). \label{eq:marginal-sure}
\end{equation}
If $d_n=|S_\tau^m|$, then $\widehat S_m(d_n)=S_\tau^m$ on the same event.
\end{theorem}

The weakest marginal signal must dominate the adaptive spot-correlation error, and the retained list must be large enough to contain the marginal support. On the concentration event, every active marginal coordinate ranks ahead of every coordinate outside $S_\tau^m$. If $x=x_n\to\infty$ while the stated conditions remain satisfied, the probability in \eqref{eq:marginal-sure} tends to one.

For the regression support, nonzero marginal scores outside $S_\tau^\beta$ can compete with active variables. For $a>0$, the conditions

\[
\min_{j\in S_\tau^\beta}|r_{j,\tau}| \ge a\ge C\varepsilon_n(x),\qquad d_n\ge\frac{4\|r_\tau\|_2^2}{a^2}
\]

give the same probability bound for $S_\tau^\beta\subseteq\widehat S_m(d_n)$. On the event with estimation error at most $a/4$, every coordinate ranked no lower than an active variable has population marginal score at least $a/2$ in absolute value. There are at most $4\|r_\tau\|_2^2/a^2$ such coordinates. We next use partial residualization to retain regression signals hidden by marginal cancellation.

\subsection{Regression-support screening}
\label{subsec:beta-theory}

Let $S=S_\tau^\beta$, $s_n=|S|$, and
\[
a_\beta=\min_{j\in S}|\theta_{j,\tau}|,
\]
with $a_\beta=\infty$ when $S=\varnothing$. The standardized explained variance satisfies

\[
\theta_\tau^\top R_\tau\theta_\tau=\frac{\beta_\tau^\top C_\tau\beta_\tau}{c_{YY,\tau}}\in[0,1].
\]

Regression-support screening requires population residual scores to remain identifiable along small active sets. The following condition gives this geometry together with deterministic envelopes for sparsity, beta-min, and explained variance.

\begin{assumption}\label{ass:sparse-geometry}
For the pre-specified model cap $q_n$, there exist deterministic sequences $\bar s_n$, $\underline a_{\beta,n}>0$, and $0\le\overline F_{\beta,n}\le1$, and a constant $\phi_->0$, such that $\Pp(G_n^\beta)\to1$, where
\begin{equation*}
\begin{aligned} G_n^\beta=\biggl\{s_n\le\bar s_n,\quad a_\beta\ge\underline a_{\beta,n},\quad \theta_\tau^\top R_\tau\theta_\tau\le\overline F_{\beta,n},\quad \inf_{\substack{U\subseteq\{1,\ldots,p_n\}\\ 1\le|U|\le q_n+\bar s_n}} \lambda_{\min}(R_{UU,\tau})\ge\phi_-\biggr\}. \end{aligned}
\end{equation*}
The choice $\overline F_{\beta,n}=1$ is always available.
\end{assumption}

The sparse eigenvalue condition has two roles. It makes every population projection used along the active-set path well defined, and it converts standardized beta-min into a residual covariance signal whenever an active variable remains unselected. On $G_n^\beta$, for a current set $A$ that does not yet contain $S$, let $B=S\setminus A$. From $r_\tau=R_\tau\theta_\tau$,
\begin{equation*}
\zeta_{B\mid A}=\left(R_{BB,\tau}-R_{BA,\tau}R_{AA,\tau}^{-1}R_{AB,\tau}\right)\theta_{B,\tau}.
\end{equation*}
The matrix in parentheses is a Schur complement of $R_{A\cup S,A\cup S,\tau}$. Assumption~\ref{ass:sparse-geometry} therefore implies
\begin{equation}
\max_{j\in S\setminus A}|\zeta_{j\mid A}| \ge\phi_-a_\beta. \label{eq:conditional-signal}
\end{equation}
Thus sparse eigenvalues and beta-min imply a residual signal without a separate partial-faithfulness condition. The next proposition isolates the deterministic stability step needed to transfer this population signal to one estimated active set. Corollary~\ref{cor:pathwise-lepski} then gives simultaneous control over all active sets encountered by the procedure.

\begin{proposition}\label{prop:partial-stability}
Fix a realization in $G_n^\beta$. Let $A\subseteq\{1,\ldots,p_n\}$ and $|A|\le q$, where $1\le q\le q_n$. Let $\widehat r\in\R^{p_n}$, and let $\widehat R_{\cdot A}$ be the requested active columns. Suppose these requested blocks are restrictions of a conceptual augmented correlation matrix
\begin{equation*}
\widehat{\mathcal R}= \begin{pmatrix} \widehat R&\widehat r\\
\widehat r^\top&\widehat R_{YY} \end{pmatrix}.
\end{equation*}
Suppose that
\[
\mathcal R_\tau\succeq0, \qquad \widehat{\mathcal R}\succeq0, \qquad \max_{1\le a\le\bar p_n}\widehat{\mathcal R}_{aa}\le1,
\]
and that the sparse-eigenvalue inequality in $G_n^\beta$ holds. Assume further that
\[
\bnorm{\widehat r-r_\tau}_\infty\le\eta_r, \qquad \bnorm{\widehat R_{\cdot A}-R_{\cdot A,\tau}}_{\max}\le\eta_R, \qquad q\eta_R\le\frac{\phi_-}{2}.
\]
When $A=\varnothing$, the requested-column condition is void and we use $\widehat\zeta_{j\mid\varnothing}=\widehat r_j$ and $\zeta_{j\mid\varnothing}=r_{j,\tau}$. If $A\ne\varnothing$, then $\widehat R_{AA}$ is invertible. In all cases,
\begin{equation*}
\sup_{j\notin A} |\widehat\zeta_{j\mid A}-\zeta_{j\mid A}| \le C_{\mathrm{pr}}(\phi_-)\left(\sqrt q\,\eta_r+q\eta_R\right),
\end{equation*}
where $C_{\mathrm{pr}}(\phi_-)$ is a finite constant depending only on $\phi_-$.
\end{proposition}

The estimator in \eqref{eq:estimated-augmented-correlation} satisfies the stated positive-semidefinite and diagonal conditions. For fixed $A$, the bound uses only $\widehat r$, $\widehat R_{\cdot A}$, and $\widehat R_{AA}$.

Define
\begin{equation*}
\Delta_{n,q}(x) =C_\Delta(\phi_-)(\sqrt q+q)\varepsilon_n(x).
\end{equation*}
Here $C_\Delta(\phi_-)$ is a sufficiently large finite constant depending only on $\phi_-$ and the constants in Theorem~\ref{thm:joint-adaptation}. The increase in population explained variance gives a sufficient number of screening updates.

\begin{theorem}\label{thm:pr-sisis}
Suppose the conditions of Corollary~\ref{cor:pathwise-lepski} and Assumption~\ref{ass:sparse-geometry} hold. Run Algorithm~\ref{alg:pr-sisis} with deterministic $m_n$ and $q_n$, and let $L_n=\lceil q_n/m_n\rceil$. Assume, for a sufficiently small constant $c_{\mathrm{stab}}>0$,
\begin{equation}
\begin{gathered} L_n\ge \left\lceil \min\left\{ \frac{4\overline F_{\beta,n}}{\phi_-^2\underline a_{\beta,n}^{\,2}},\, \frac{16\bar s_n}{\phi_-} \log\!\left(1+\frac{8\bar s_n}{\phi_-}\right) \right\} \right\rceil,\\ q_n\le k_{\min},\qquad q_n\varepsilon_n(x)\le c_{\mathrm{stab}}\phi_-,\qquad \phi_-\underline a_{\beta,n}\ge4\Delta_{n,q_n}(x). \end{gathered} \label{eq:block-size}
\end{equation}
Then
\begin{equation}
\Pp\{S_\tau^\beta\subseteq\widehat S_{\mathrm{PR}}\} \ge1-Ce^{-cx}-o(1). \label{eq:pr-sure}
\end{equation}
On the same event, every sample correlation submatrix inverted by the algorithm is nonsingular. The $o(1)$ term includes $\Pp\{(G_n^\beta)^c\}$.
\end{theorem}

The two terms in the minimum in \eqref{eq:block-size} give alternative update counts based on explained variance and sparsity. The other conditions limit the active dimension to the local sample size, ensure stable matrix inversion, and require the residual signal to exceed the uniform score error. For fixed $m_n$, the second bound permits $q_n$ to be a sufficiently large multiple of $\bar s_n\log(e+\bar s_n)$, subject to these conditions. If $\|R_\tau\|_{\mathrm{op}}\le\phi_+$ with probability tending to one for fixed $\phi_+<\infty$, the second term can be replaced by
\[
\frac{16\bar s_n}{\phi_-} \log\!\left(1+\frac{8\phi_+}{\phi_-}\right).
\]
The required number of updates is then of order $\bar s_n$ for fixed $m_n$.

\subsection{List-size minimax lower bound}
\label{subsec:list-lower-bound}

The preceding results give sufficient conditions for screening. We now derive a necessary signal scale for retaining the spot regression support in a list of prescribed size.

For $0<a\le1/2$, set

\[
h_a=2\left(\frac{a}{L_\Sigma}\right)^{1/\alpha}, \qquad f_a(t)=a\left(1-\frac{|t-\tau|}{h_a}\right)_+,
\]

where $u_+=\max(u,0)$. Let $P_{j,a}^{(n)}$ be the law of $\mathcal O_n=\{(X_{t_i^n},Y_{t_i^n}):i=0,\ldots,n\}$ under

\[
\dd X_t=\dd B_t^X,\qquad \dd Y_t=f_a(t)\,\dd X_{j,t} +\sqrt{1-f_a(t)^2}\,\dd W_t^\varepsilon,
\]

with $X_0=Y_0=0$, where $B^X$ and $W^\varepsilon$ are independent standard Brownian motions of dimensions $p_n$ and one. With $e_j$ denoting the $j$th coordinate vector, these models satisfy

\[
R_t=I_{p_n},\qquad c_{YY,t}=1,\qquad \theta_t=r_t=f_a(t)e_j, \qquad S_\tau^\beta=S_\tau^m=\{j\}.
\]

They satisfy Assumptions~\ref{ass:jump} and~\ref{ass:local-characteristics}. The local regularity calculation is given in the Supplement. Write $\calP_{\alpha,n}(a)=\{P_{j,a}^{(n)}:1\le j\le p_n\}$.

\begin{theorem}\label{thm:list-lower-bound}
Let $(a_n)$ be a deterministic positive sequence with $a_n\le1/2$ and $[\tau-h_{a_n},\tau+h_{a_n}]\subset(0,1)$, and let $1\le d_n<p_n/4$. There exist constants $c_0,c_1>0$, depending only on $\alpha$ and $L_\Sigma$, such that, if
\[
a_n\le c_1\left\{\frac{\log(p_n/d_n)}{n}\right\}^{\alpha/(2\alpha+1)},
\]
then
\[
\inf_{\widehat S:\,|\widehat S|\le d_n\ \mathrm{a.s.}} \sup_{P\in\calP_{\alpha,n}(a_n)} P\{S_\tau^\beta\nsubseteq\widehat S\}\ge c_0.
\]
The infimum is over all measurable lists based on $\mathcal O_n$, with the size constraint holding under every law in the class. The same bound holds with $S_\tau^m$ in place of $S_\tau^\beta$.
\end{theorem}

\begin{remark}
The bound applies to any larger model class containing these alternatives. When $\log(p_n/d_n)\asymp\log p_n$, it matches the localization and multiplicity order of Theorem~\ref{thm:marginal-screening}.
\end{remark}

Figure~\ref{fig:parameter-theory} summarizes the theoretical tradeoffs. Write $\log p_n\asymp n^\gamma$, $k_n\asymp n^\kappa$, $q_n\asymp n^\xi$, and $d_n=p_n^\delta$, with $0\le\delta<1$ and confidence and grid terms that do not change the leading exponent of $\Lambda_n(x)$. Panel (a) shows the balance at $\kappa=(2\alpha+\gamma)/(2\alpha+1)$. Panel (b) gives the sufficient stability region $\xi<\alpha(1-\gamma)/(2\alpha+1)$. The remaining screening conditions are in Theorem~\ref{thm:pr-sisis}. Panel (c) shows the lower-bound factor $(1-\delta)^{\alpha/(2\alpha+1)}$. The curves use $\alpha=1/2$ as an illustration. Generic diffusion-driven covariance paths satisfy the stated local regularity for fixed $\alpha<1/2$.

\begin{figure}[!ht]
\centering
\begin{tikzpicture}
\def\alphavis{0.5}
\begin{groupplot}[group style={group size=3 by 1,horizontal sep=0.85cm},width=0.265\textwidth,height=0.215\textwidth,scale only axis,axis lines=left,tick label style={font=\scriptsize},label style={font=\scriptsize},title style={font=\scriptsize},every axis plot/.append style={line width=0.9pt},clip=false]
\nextgroupplot[title={(a) Dimension and localization},xlabel={$\gamma$},ylabel={$\kappa$},xmin=0,xmax=1,ymin=0,ymax=1,xtick={0,0.5,1},ytick={0,0.5,1}]
\addplot[densely dashed,black,domain=0:1] {x};
\addplot[black,thick,domain=0:1] {(2*\alphavis+x)/(2*\alphavis+1)};
\node[font=\tiny,anchor=north west,align=left] at (axis description cs:0.03,0.97) {$\kappa^\star(\gamma)=\dfrac{2\alpha+\gamma}{2\alpha+1}$};
\node[font=\tiny] at (axis cs:0.18,0.79) {bias};
\node[font=\tiny] at (axis cs:0.46,0.60) {sampling};
\node[font=\tiny] at (axis cs:0.68,0.18) {sampling failure};
\nextgroupplot[title={(b) PR stability},xlabel={$\gamma$},ylabel={$\xi$},xmin=0,xmax=1,ymin=0,ymax=0.30,xtick={0,0.5,1},ytick={0,0.15,0.30},yticklabel style={/pgf/number format/fixed,/pgf/number format/precision=2,font=\scriptsize}]
\addplot[black,thick,domain=0:1] {\alphavis*(1-x)/(2*\alphavis+1)};
\node[font=\tiny,anchor=north west,align=left] at (axis description cs:0.03,0.97) {$\xi=\dfrac{\alpha(1-\gamma)}{2\alpha+1}$};
\node[font=\tiny] at (axis cs:0.42,0.06) {$q_n\varepsilon_n(x)\to0$};
\nextgroupplot[title={(c) Retained-list effect},xlabel={$\delta$},ylabel={rel.\ factor},ylabel style={font=\scriptsize},xmin=0,xmax=1,ymin=0,ymax=1.05,xtick={0,0.5,1},ytick={0,0.5,1}]
\addplot[black,thick,domain=0:1,samples=200,variable=\t] ({1-\t^4},{\t});
\node[font=\tiny,align=center] at (axis cs:0.50,0.31) {$(1-\delta)^{\alpha/(2\alpha+1)}$};
\end{groupplot}
\end{tikzpicture}
\caption{Illustrative exponent regions under $\log p_n\asymp n^\gamma$, $k_n\asymp n^\kappa$, $q_n\asymp n^\xi$, and $d_n=p_n^\delta$. Panel (a) shows the dimension and localization balance. Panel (b) shows the sufficient stability region. Panel (c) shows the retained-list effect. The curves use $\alpha=1/2$.}
\label{fig:parameter-theory}
\end{figure}

The jump-truncation conditions further require $\gamma<1-2\varpi$ and $\kappa-\gamma<\varpi(2-r_0)$ when $\kappa>\gamma$, with the latter checked at the largest candidate window. At the balance, this becomes $2\alpha(1-\gamma)/(2\alpha+1)<\varpi(2-r_0)$. Together these inequalities require $r_0<2/(2\alpha+1)$ for that balance to satisfy the stated compatibility conditions.

\FloatBarrier

% ======================================================================
\section{Simulations}
\label{sec:simulation}
% ======================================================================

\subsection{Simulation designs}
\label{subsec:reference-designs}

The simulations study screening under marginal cancellation and window selection under changing local regression relationships. The reference experiments examine screening performance across signal strengths and dimensions. Each replication consists of one path on $[0,1]$, observed at $n=390$ one-minute increments, with target time $\tau=1/2$ and support $S=\{1,\ldots,5\}$. In the regular cancellation designs, denoted LC-R for $p=50$ and HC-R for $p=500$, the correlation matrix $R$ is block diagonal: its first three $2\times2$ blocks have correlations $0.60$, $0.55$, and $0.60$, while the remaining inactive coordinates form an AR(1) block with parameter $0.10$. The coefficient direction is
\[
b=(0.70,-0.70,0.60,-0.60,0.50,0,\ldots,0)^\top.
\]
We set $\beta_0=\lambda b$ and $c_\varepsilon=1$. For each prescribed standardized beta-min, $\lambda$ is chosen to satisfy
\[
a_\beta=\frac{0.50\lambda}{\{\lambda^2b^\top Rb+c_\varepsilon\}^{1/2}},\qquad a_\beta\in\{0.15,0.20,0.25,0.30,0.35,0.40,0.45\}.
\]
The coefficient direction and covariance matrix are held fixed across signal levels, and common random numbers are used within each configuration. The weak-dependence design LW has $p=50$, $R_{jk}=0.10^{|j-k|}$, $c_\varepsilon=0.50$, and active coefficients $(0.60,0.55,0.50,0.45,0.40)$, which give $a_\beta\approx0.283$. For its signal curve, we scale this coefficient vector while holding the residual variance fixed, using 0.10, 0.15, 0.20, 0.25, the reference value $0.283\ldots$, 0.30, and 0.32. The design HC-S modifies HC-R by setting $R_{56}=R_{65}=0.80$, with coordinate 6 remaining inactive, and uses the same coefficient direction and signal grid as HC-R.

The reference designs use the same stochastic volatility, leverage, intraday periodicity, time-varying sparse coefficients, and finite-activity jumps. Let $v_0=1$ and define
\[
s(t)=1+0.15\sin(2\pi t), \qquad g(t)=1+0.10\sin\{2\pi(t-\tau)\}.
\]
The continuous-time data-generating process is
\begin{align*}
\dd v_t &=\frac{6}{252}(1-v_t)\dd t +\frac{0.50}{\sqrt{252}}\sqrt{v_t}\dd U_t+\dd J_t^v,\\ \dd X_t &=s(t)\sqrt{v_t}R_t^{1/2}\dd B_t+\dd J_t^X,\\ \dd Y_t &=\{g(t)\beta_0\}^\top\dd X_t +s(t)\sqrt{c_\varepsilon v_t}\dd W_t^\varepsilon+\dd J_t^\varepsilon.
\end{align*}
Here $R_t^{1/2}$ is the lower triangular covariance factor. The standardized covariate innovations have covariance $R_t$ and covariance vector $(-0.30,0,\ldots,0)^\top$ with the volatility innovation. The residual Brownian motion is independent of these innovations. The jump components include common and coordinate-specific covariate jumps, residual jumps, and a volatility jump. Their distributions and intensities, together with the Euler implementation, are given in the Supplement.

At the target, $s(\tau)=g(\tau)=1$, so the design parameters $\beta_0$, $R_\tau$, and $c_\varepsilon$ determine the standardized beta-min used to index the designs. Away from the target, $s(t)$ produces a common intraday scale and $g(t)$ varies the sparse coefficient path. The stochastic volatility factor adds pathwise scale variation, while the leverage and jump components retain the high-frequency features covered by the model. The reference screening designs take $R_t=R$ near the target. The dynamic designs in Section~\ref{subsec:dynamic-window-results} vary the coefficient and correlation paths while retaining the volatility, leverage, and jump structure.

We use the candidate grid $\mathcal K_n=\{32,40,50,63,79,99,124,155\}$ and set $C_{\mathrm L}=0.40$ and $x=0$. Coordinatewise truncation uses the full-day bipower scale with multiplier $4$ and exponent $0.47$, and the covariance diagonal is floored at $10^{-12}$ before standardization. For $d=10$, PR-SISIS uses $m_n=2$ and $q_n=10$. For the HC-R comparison with $d=20$, it uses $m_n=2$ and $q_n=20$. We use 400 replications for each cancellation design and 500 replications for LW. Ties are resolved by increasing coordinate index.

For a retained set $\widehat S(d)$, the percentage of active coordinates recovered is
\[
\mathrm{Recovery}=100\,\frac{|S\cap\widehat S(d)|}{|S|}.
\]
Sure is the percentage of replications in which $S\subseteq\widehat S(d)$. We also report the average position of the last active coordinate in the complete ordering. M-SIS orders all coordinates by absolute marginal scores. PR-SISIS concatenates the selected blocks in update order and appends the remaining coordinates in decreasing order of their final absolute partial-residual scores.

\subsection{Signal strength and screening performance}

\FloatBarrier
\begin{figure}[!ht]
\centering
\includegraphics[width=\textwidth]{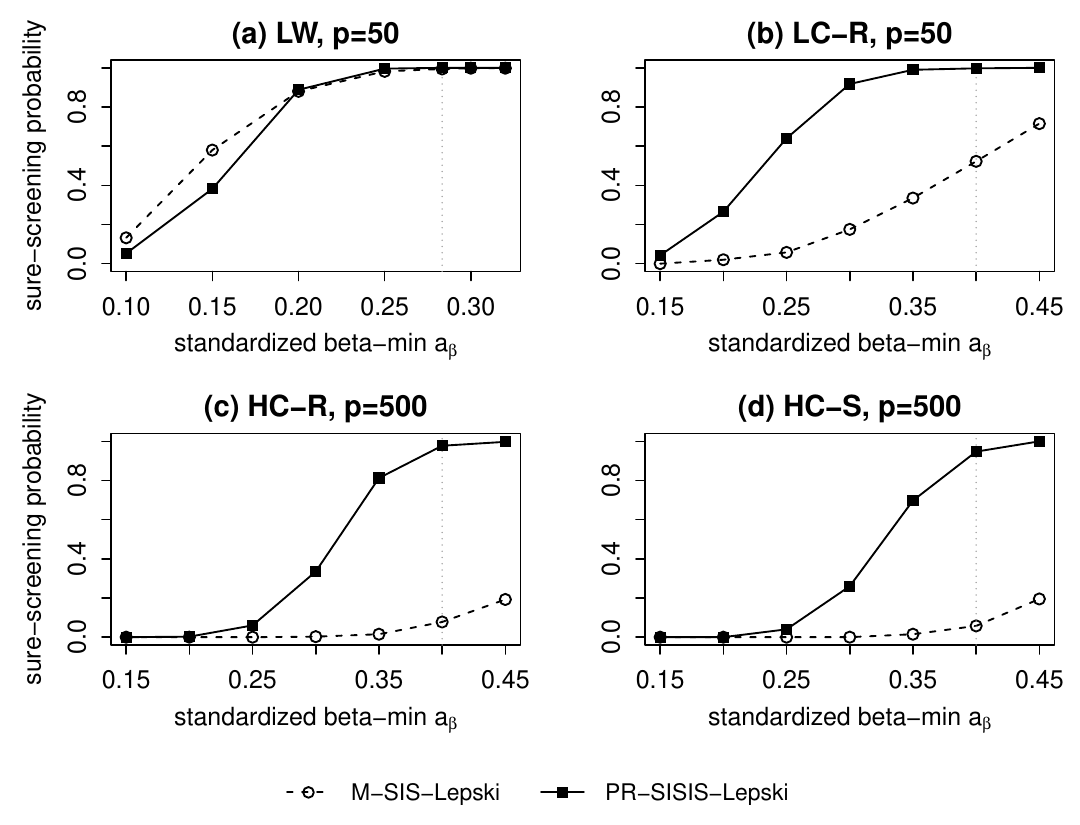}
\caption{Sure-screening probabilities in the four reference configurations with ten retained coordinates. Solid squares denote PR-SISIS-Lepski and dashed circles denote M-SIS-Lepski. Vertical lines mark the reference signal levels in Table~\ref{tab:simulation-reference}.}
\label{fig:signal-strength}
\end{figure}

Figure~\ref{fig:signal-strength} reports the sure-screening probabilities across signal levels. At the weakest signals in LW, marginal screening retains the support more often. Both procedures approach full retention as the signal strengthens. Under cancellation, partial residualization reaches high retention at weaker signals. Increasing the dimension shifts the transition toward stronger signals, and the inactive surrogate in HC-S makes marginal ranking more difficult.

\begin{table}[!ht]
\centering
\caption{Screening performance in the reference configurations}
\label{tab:simulation-reference}
\scriptsize
\setlength{\tabcolsep}{3.2pt}
\renewcommand{\arraystretch}{1.10}
\begin{tabular*}{\textwidth}{@{\extracolsep{\fill}}lr*{6}{r}@{}}
\toprule
Config. & $p$
& \multicolumn{2}{c}{Sure (\%)}
& \multicolumn{2}{c}{Recovered (\%)}
& \multicolumn{2}{c}{Avg. last-active rank}\\
\cmidrule(lr){3-4}
\cmidrule(lr){5-6}
\cmidrule(lr){7-8}
& & M-SIS & PR-SISIS
& M-SIS & PR-SISIS
& M-SIS & PR-SISIS\\
\midrule
\multicolumn{8}{@{}l}{\textit{Panel A: Retained size $d=10$}}\\
\addlinespace[2pt]
LW & 50 & 99.4 & 100.0 & 99.9 & 100.0 & 5.2 & 5.0\\
LC-R & 50 & 52.3 & 99.8 & 89.4 & 100.0 & 13.0 & 5.4\\
HC-R & 500 & 7.8 & 97.8 & 69.9 & 99.1 & 76.9 & 6.5\\
HC-S & 500 & 5.8 & 94.8 & 68.7 & 98.4 & 77.8 & 8.3\\
\midrule
\multicolumn{8}{@{}l}{\textit{Panel B: Retained size $d=20$}}\\
\addlinespace[2pt]
HC-R & 500 & 25.8 & 99.3 & 80.2 & 99.7 & 76.9 & 6.8\\
\bottomrule
\end{tabular*}

\smallskip
\begin{minipage}{\textwidth}
\footnotesize
\raggedright
\textit{Note.} For LC-R, HC-R, and HC-S, entries are Monte Carlo averages at $a_\beta=0.40$. For LW, $a_\beta\approx0.283$. Sure and Recovered are percentages.
\end{minipage}
\end{table}

Table~\ref{tab:simulation-reference} compares screening under weak dependence and marginal cancellation. In the weak-dependence reference configuration, the active coordinates have strong marginal scores and the procedures perform similarly. Under cancellation, partial residualization retains the support more often and improves the rank of the last active coordinate. The inactive surrogate in HC-S makes ranking more difficult, while a larger retained list reduces omissions. Section~\ref{sec:supp-jump-truncation} of the Supplement examines truncation and tuning.

\subsection{Window selection under changing local relationships}
\label{subsec:dynamic-window-results}

We next vary conditional dependence while keeping the marginal vector fixed:
\[
r_t=r_\tau=(a,0,0.32,0.32,0.40,0,\ldots,0)^\top,\qquad \beta_t=R_t^{-1}r_\tau.
\]
Coordinate 2 is marginally hidden. Early, Middle, and Late differ in the reversal time of the first-pair correlation, with a second correlation block changing 30 minutes earlier. Stable holds both blocks fixed. All configurations have target beta-min $a_\beta=0.40$, with conditional scores and transition times specified in Section~\ref{subsec:supp-benchmark-dgp}.

Figure~\ref{fig:dynamic-screening-windows} compares the window rules at $(p,n)=(100,390),(500,390)$, and $(500,1560)$, with $m_n=2$ and $q_n=d=10$. Candidate durations range from 32 to 155 minutes, with $k=nH/390$ increments for a duration of $H$ minutes.

Pathwise PR-SISIS retains the support with probability close to one across these settings. Short windows lose precision, while long windows can average reversing correlations and weaken the hidden residual score $-a\overline\rho(k)$, where $\overline\rho(k)$ is the variance-weighted window average of the first-pair correlation. Comparing the active correlation columns allows a shorter window. Intermediate fixed windows also perform well, with 99.8--100.0\% retention at 79 minutes.

\begin{figure}[!htb]
\centering
\includegraphics[width=0.94\textwidth]{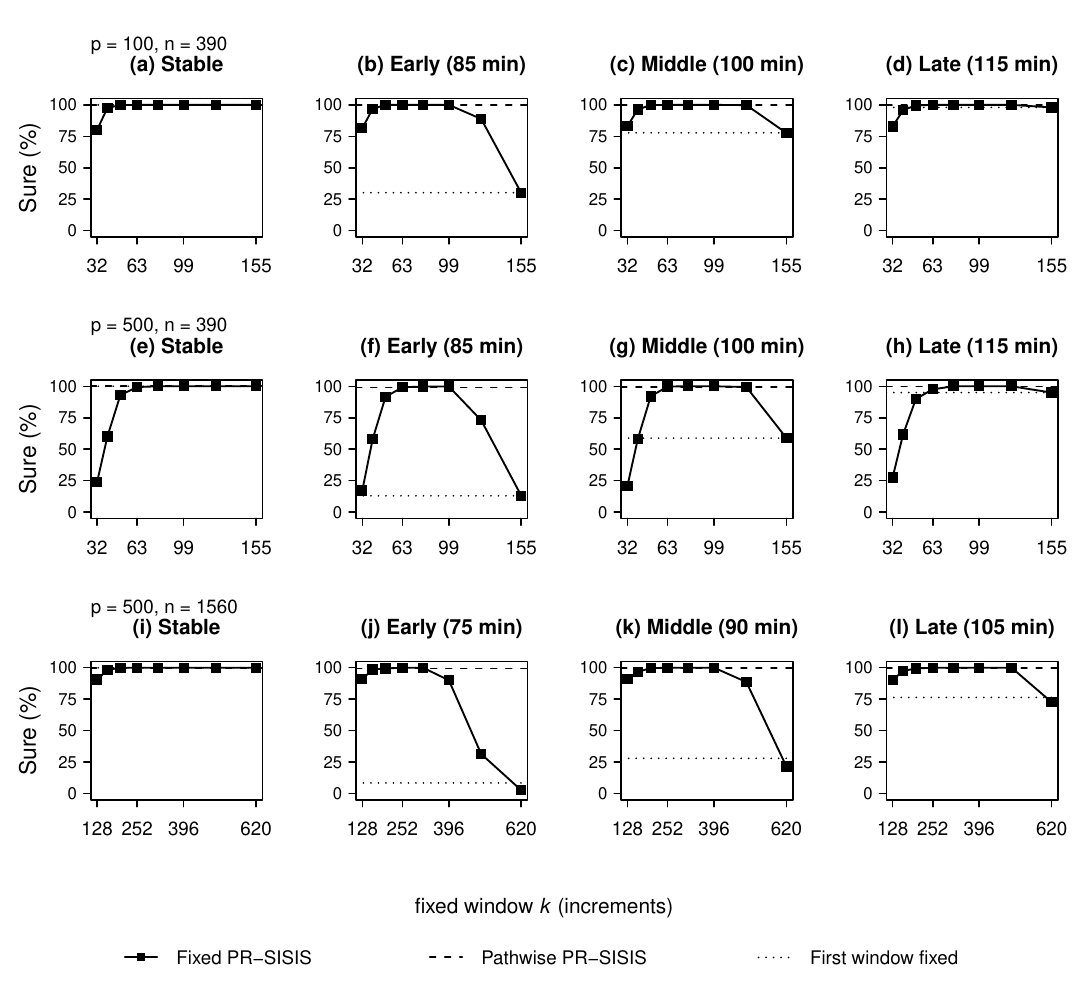}
\caption{Sure-screening percentages with target standardized beta-min $a_\beta=0.40$ and ten retained coordinates. Fixed PR-SISIS uses the window on the horizontal axis throughout the updates. First window fixed holds the initial marginally selected window fixed, whereas pathwise PR-SISIS reselects it at each update. The horizontal axis is in increments. A duration of 155 minutes corresponds to 155 increments for $n=390$ and 620 for $n=1560$.}
\label{fig:dynamic-screening-windows}
\end{figure}

\newpage

\section{Empirical illustration}
\label{sec:empirical}

\subsection{Data and implementation}

This high-frequency factor-zoo panel has been used to study intraday market-return predictability \cite{AletiBollerslevSiggaard2025}. Here we use it to illustrate fixed-time local screening. The response is the Fama and French market portfolio, and the remaining 271 one-minute factor and industry portfolio returns are candidate covariates. After excluding the two shortened sessions, we retain 251 complete regular U.S. trading days in 2020. Removing the adjusted overnight observation leaves 390 one-minute decimal simple returns per day, with endpoints from 09:31 through 16:00. Adjacent bipower products are computed within each day. No demeaning, winsorization, normalization, or factor preselection is applied. The primary target is 10:00 New York time, whose first forward return ends at 10:01.

We analyze each day as a separate high-frequency path. This preserves the fixed-time interpretation of the spot target and allows the selected window to respond to that day's local covariance path. The local covariance sums use only returns after 10:00. We use the simulation window grid
\[
\{32,40,50,63,79,99,124,155\},
\]
with $C_L=0.40$, $x=0$, and the same truncation and variance-floor rules. Both methods retain 10 factors. PR-SISIS uses $m_n=2$ and $q_n=10$. Selection frequency is the fraction of retained days on which a factor enters the top-10 list. We also compare the pathwise lists with PR-SISIS at every fixed candidate window and with the initial selected window held fixed. The same comparisons are repeated at 11:30 and 13:00, starting from the empty active set at each target. The heatmap averages the daily 155-minute factor correlation matrices at 10:00, using the same coordinatewise truncation, and groups the 271 factors by their original Cluster classification.

\subsection{Results}

Table~\ref{tab:empirical-top-factors} shows substantial differences between marginal and partial-residual screening. Pharmaceutical products is retained on 51.4\% of days by PR-SISIS and 9.2\% by M-SIS. The corresponding frequencies for Utilities are 25.1\% and 1.2\%. Electronic equipment and Retail are less frequently retained by PR-SISIS. The two daily lists share 2.80 factors on average, whereas pathwise PR-SISIS and fixed-window PR-SISIS at 155 minutes share 9.77. At this target, the differences are therefore associated mainly with partial residualization rather than window selection.

\FloatBarrier
\begin{table}[!htbp]
\centering
\caption{Most frequently retained factors at the 10:00 target}
\label{tab:empirical-top-factors}
\scriptsize
\setlength{\tabcolsep}{2pt}
\renewcommand{\arraystretch}{1.15}
\begin{tabular*}{\textwidth}{@{\extracolsep{\fill}}l*{10}{c}@{}}
\toprule
Factor & \shortstack{Business\\services} & \shortstack{Electronic\\equipment} & \shortstack{Pharmaceutical\\products} & Retail & Wholesale & Insurance & Utilities & Trading & Banking & \shortstack{Book\\leverage}\\
\midrule
Cluster & I & I & I & I & I & I & I & I & I & L\\
Pathwise PR & 0.932 & 0.590 & 0.514 & 0.375 & 0.267 & 0.255 & 0.251 & 0.227 & 0.203 & 0.191\\
M-SIS & 0.992 & 0.944 & 0.092 & 0.765 & 0.713 & 0.195 & 0.012 & 0.629 & 0.251 & 0.000\\
\bottomrule
\end{tabular*}
\smallskip
\begin{minipage}{\textwidth}\footnotesize\raggedright
\textit{Note.} Frequencies are fractions of the 251 retained days. Columns follow the pathwise PR-SISIS frequency ranking. I denotes Industry and L denotes Leverage. Each list contains ten factors.
\end{minipage}
\end{table}

 Figure~\ref{fig:empirical-clusters} relates these differences to the dependence structure of the candidate panel. Panel B shows positive correlations within the Industry block and correlations across categories. Industry portfolios account for 82.6\% of M-SIS positions but 55.2\% of PR-SISIS positions, with larger shares allocated to Quality, Investment, and Leverage under PR-SISIS. Marginal scores can reflect linear contributions shared with other covariates, whereas partial residualization ranks the association remaining after projection on the selected variables. These patterns are consistent with the distinction between marginal and regression relevance in $r_\tau=R_\tau\theta_\tau$. The Industry share remains 55.1\% with a fixed 155-minute PR window, further indicating that partial residualization drives the main composition change at this target.

\begin{figure}[!htbp]
\centering
\includegraphics[width=\textwidth]{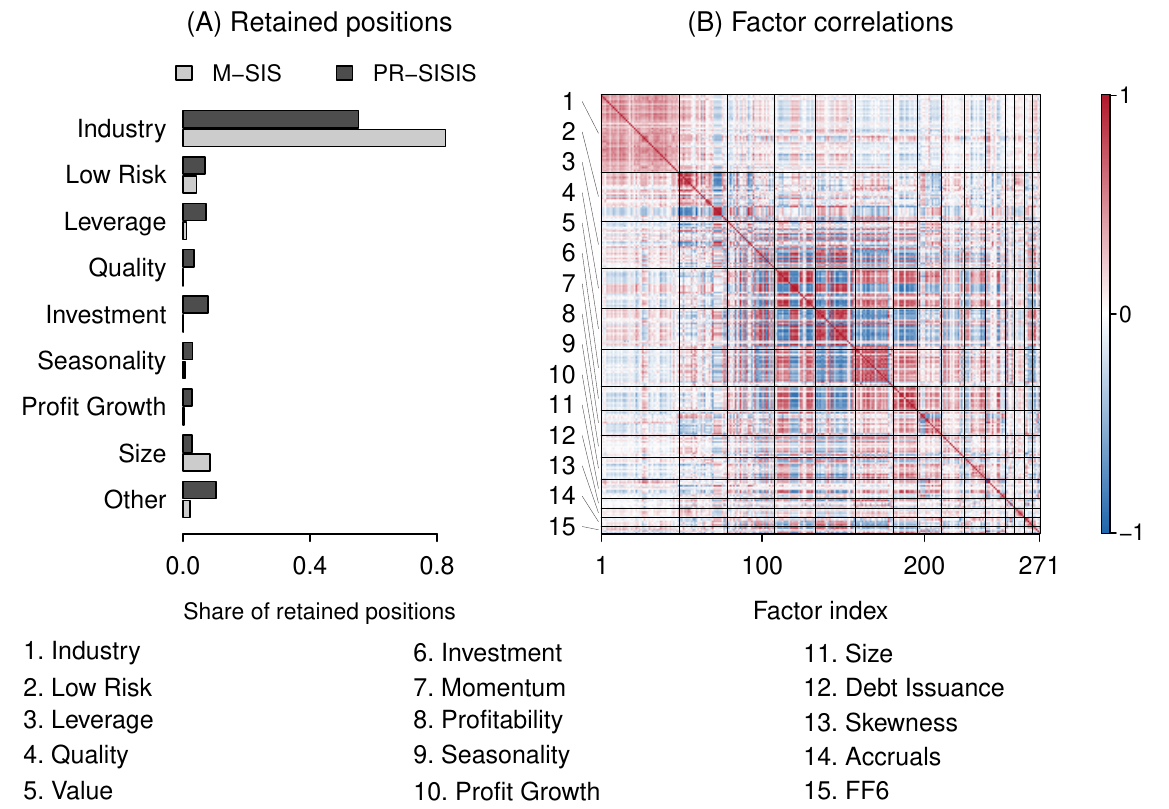}
\caption{Factor composition and dependence at 10:00. Panel A reports category shares of the $251\times10$ retained positions. Other combines the categories not shown separately. Panel B averages the daily factor correlation matrices estimated over the 155-minute forward window using coordinatewise truncation. Factors are ordered by Cluster, with lines marking category boundaries.}
\label{fig:empirical-clusters}
\end{figure}

Table~\ref{tab:empirical-window-comparison} shows that window choice has a larger effect at later targets. The shorter fixed windows produce lists with less overlap with pathwise PR-SISIS, while the longest fixed window and the first selected window held fixed give closer lists. Shorter update windows also occur more often later in the day, when updating the correlation inputs has a greater effect on the retained set.

\begin{table}[!htbp]
\centering
\caption{Window selection and retained-list overlap across targets}
\label{tab:empirical-window-comparison}
\small
\setlength{\tabcolsep}{4pt}
\begin{tabular*}{\textwidth}{@{\extracolsep{\fill}}lrrrrrr@{}}
\toprule
& & \multicolumn{4}{c}{Mean overlap with fixed $k$} & \\
\cmidrule(lr){3-6}
Target & \shortstack{Short update\\(days)} & 32 & 79 & 124 & 155 & \shortstack{First window\\fixed}\\
\midrule
10:00 & 15 & 2.63 & 4.64 & 6.51 & 9.77 & 9.87\\
11:30 & 45 & 2.43 & 4.04 & 5.90 & 9.27 & 9.71\\
13:00 & 101 & 2.17 & 4.09 & 5.71 & 8.23 & 9.17\\
\bottomrule
\end{tabular*}
\smallskip
\begin{minipage}{\textwidth}\footnotesize\raggedright
\textit{Note.} Each target uses all 251 days. Short update counts days with a window below 155 minutes in at least one of the five selection updates. Overlap is the mean number of common factors with pathwise PR-SISIS, out of ten. First window fixed holds the initial selected window throughout the updates. Targets on the same day are dependent.
\end{minipage}
\end{table}

\FloatBarrier
\section{Conclusions}
\label{sec:conclusions}

We developed screening methods for continuous-time regression observed at high frequency. Uniform entrywise bounds over a window grid yield adaptive correlation estimates at rate $\{(\log p)/n\}^{\alpha/(2\alpha+1)}$, up to confidence and grid factors. These bounds give sure screening for the marginal support and, under sparse eigenvalues and standardized beta-min, for the regression support. A list-size minimax lower bound applies to both targets and matches the localization and multiplicity order of the marginal-screening bound.

Partial residualization addresses cancellation and surrogate effects that can defeat marginal screening. The dynamic designs show that active-column comparisons can reveal changes absent from marginal correlations, while the gains in support retention depend on the change timing and local sample size. In the empirical illustration, marginal and partial-residual screening yield different recurrent factor rankings, and the effect of window selection varies across intraday targets. Extensions to asynchronous observations, alternative local weighting schemes, and post-screening estimation remain for future work.

% ======================================================================
% References
% ======================================================================
\begin{supplement}
\stitle{Supplement to ``Ultra-high-dimensional spot support screening in continuous-time regression''}
\sdescription{This supplement contains proofs, implementation details, and additional simulations.}
\end{supplement}

\begin{funding}
This work was supported by the National Natural Science Foundation of China (NSFC, 12501361).
\end{funding}

% ======================================================================
% SUPPLEMENTARY MATERIAL
% ======================================================================
% Local layout for the Supplementary Material.

\newgeometry{margin=1in} \microtypesetup{activate=true} \hypersetup{hidelinks,pageanchor=false}
\begingroup
\makeatletter

% === article[11pt]-like type sizes (local to the supplement) ===
\renewcommand\normalsize{%
   \@setfontsize\normalsize\@xipt{13.6}%
\abovedisplayskip 11\p@ \@plus3\p@ \@minus6\p@ \abovedisplayshortskip \z@ \@plus3\p@ \belowdisplayshortskip 6.5\p@ \@plus3.5\p@ \@minus3\p@ \belowdisplayskip \abovedisplayskip}
\renewcommand\small{%
\@setfontsize\small\@xpt\@xiipt \abovedisplayskip 10\p@ \@plus2\p@ \@minus5\p@ \abovedisplayshortskip \z@ \@plus3\p@ \belowdisplayshortskip 6\p@ \@plus3\p@ \@minus3\p@ \belowdisplayskip \abovedisplayskip}
\renewcommand\footnotesize{%
   \@setfontsize\footnotesize\@ixpt{11}%
\abovedisplayskip 8\p@ \@plus2\p@ \@minus4\p@ \abovedisplayshortskip \z@ \@plus\p@ \belowdisplayshortskip 4\p@ \@plus2\p@ \@minus2\p@ \belowdisplayskip \abovedisplayskip}
\renewcommand\scriptsize{\@setfontsize\scriptsize\@viiipt{9.5}}
\renewcommand\tiny{\@setfontsize\tiny\@vipt\@viipt}
\renewcommand\large{\@setfontsize\large\@xiipt{14}}
\renewcommand\Large{\@setfontsize\Large\@xivpt{18}}
\renewcommand\LARGE{\@setfontsize\LARGE\@xviipt{22}}
\renewcommand\huge{\@setfontsize\huge\@xxpt{25}}
\renewcommand\Huge{\@setfontsize\Huge\@xxvpt{30}}
\normalsize
\setlength\parindent{17\p@}
\setlength\parskip{0\p@ \@plus \p@}
\setlength\headheight{12\p@}
\setlength\headsep{25\p@}
\setlength\footskip{30\p@}
\raggedbottom

% === article-like page style ===
\def\ps@suppplain{%
  \let\@mkboth\@gobbletwo
  \let\@oddhead\@empty
  \let\@evenhead\@empty
  \def\@oddfoot{\reset@font\hfil\thepage\hfil}%
  \let\@evenfoot\@oddfoot}
\pagestyle{suppplain}
\setcounter{page}{1}

% === article-like section headings ===
\renewcommand\section{\@startsection {section}{1}{\z@}%
                                   {-3.5ex \@plus -1ex \@minus -.2ex}%
                                   {2.3ex \@plus.2ex}%
{\normalfont\Large\bfseries}}
\renewcommand\subsection{\@startsection{subsection}{2}{\z@}%
                                     {-3.25ex\@plus -1ex \@minus -.2ex}%
                                     {1.5ex \@plus .2ex}%
{\normalfont\large\bfseries}}
\renewcommand\subsubsection{\@startsection{subsubsection}{3}{\z@}%
                                     {-3.25ex\@plus -1ex \@minus -.2ex}%
                                     {1.5ex \@plus .2ex}%
{\normalfont\normalsize\bfseries}}
\renewcommand\paragraph{\@startsection{paragraph}{4}{\z@}%
                                    {3.25ex \@plus1ex \@minus.2ex}%
                                    {-1em}%
{\normalfont\normalsize\bfseries}}
\renewcommand\subparagraph{\@startsection{subparagraph}{5}{\parindent}%
                                       {3.25ex \@plus1ex \@minus .2ex}%
                                       {-1em}%
{\normalfont\normalsize\bfseries}}

% === article-like captions ===
\setlength\abovecaptionskip{10\p@}
\setlength\belowcaptionskip{0\p@}
\long\def\@makecaption#1#2{%
\vskip\abovecaptionskip
  \sbox\@tempboxa{#1: #2}%
\ifdim \wd\@tempboxa >\hsize #1: #2\par \else \global \@minipagefalse
    \hb@xt@\hsize{\hfil\box\@tempboxa\hfil}%
\fi \vskip\belowcaptionskip}

% === independent supplement numbering ===
\setcounter{section}{0}
\setcounter{subsection}{0}
\setcounter{subsubsection}{0}
\setcounter{equation}{0}
\setcounter{figure}{0}
\setcounter{table}{0}
\renewcommand\thesection{\arabic{section}}
\renewcommand\thesubsection{\thesection.\arabic{subsection}}
\renewcommand\thesubsubsection{\thesubsection.\arabic{subsubsection}}
\numberwithin{equation}{section}
\renewcommand\theequation{S.\thesection.\arabic{equation}}
\@removefromreset{figure}{section} \@removefromreset{table}{section}
\renewcommand\thetable{S.\arabic{table}}
\renewcommand\thefigure{S.\arabic{figure}}

% Independent supplementary lemma numbering.
\theoremstyle{plain}
\newtheorem{supplemma}{Lemma}
\renewcommand\thesupplemma{S.\arabic{supplemma}}
\setcounter{supplemma}{0}
\providecommand*\theHsupplemma{}
\renewcommand*\theHsupplemma{supp.lemma.\arabic{supplemma}}

% Unique Hyperref destination names after resetting visible counters.
\providecommand*\theHsection{}
\providecommand*\theHsubsection{}
\providecommand*\theHsubsubsection{}
\providecommand*\theHequation{}
\providecommand*\theHfigure{}
\providecommand*\theHtable{}
\renewcommand*\theHsection{supp.\arabic{section}}
\renewcommand*\theHsubsection{supp.\arabic{section}.\arabic{subsection}}
\renewcommand*\theHsubsubsection{supp.\arabic{section}.\arabic{subsection}.\arabic{subsubsection}}
\renewcommand*\theHequation{supp.\arabic{section}.\arabic{equation}}
\renewcommand*\theHfigure{supp.\arabic{figure}}
\renewcommand*\theHtable{supp.\arabic{table}}

% Supplement macros.
\providecommand{\Prb}{\mathbb{P}}
\providecommand{\1}{\mathbf{1}}
\providecommand{\KL}{\operatorname{KL}}
\providecommand{\op}{\mathrm{op}}
\providecommand{\maxn}{\mathrm{max}}
\providecommand{\cK}{\mathcal{K}_n}
\providecommand{\deltan}{\Delta_n}
\graphicspath{{./}}

% Natbib is loaded by the main paper. Make its bibliography heading use the
% article-style section command defined above, its counter still resets to 1.
\renewcommand\bibsection{\section*{References}}

% Article-like supplement abstract, without invoking imsart's front-matter abstract.
\newenvironment{suppabstract}{%
  \small
  \begin{center}%
    {\bfseries Abstract\vspace{-.5em}\vspace{\z@}}%
  \end{center}%
\list{}{\leftmargin 2.5em\rightmargin 2.5em \listparindent 1.5em\itemindent 1.5em
           \parsep \z@ \@plus\p@}%
  \item\relax
}{\endlist}

\makeatother

% Supplement title.
\phantomsection \pdfbookmark[0]{Supplementary Material}{supplementary-material} \null \vskip 2em
\begin{center}
{\LARGE Supplement to ``Ultra-high-dimensional spot support screening in continuous-time regression''\par} \vskip 1.5em {\large
    \begin{tabular}[t]{c}
      Haibin Zhu\\
      Department of Statistics and Data Science, School of Economics, Jinan University\\
haibinzhu@jnu.edu.cn
    \end{tabular}\par}
\vskip 1em
\end{center}
\par \vskip 1.5em
\thispagestyle{suppplain}

\begin{suppabstract}
This supplement contains proofs of the main results, a factor-type sufficient condition for Assumption~\ref{ass:local-characteristics}, computational and simulation details, and additional simulations. All notation, assumptions, and result numbers refer to the main paper. Throughout the supplement, the symbols $C,c,c_1,c_2,\ldots$ denote positive constants whose values may change from line to line.
\end{suppabstract}

\section{Proof of Theorem~\ref{thm:joint-adaptation} and Corollary~\ref{cor:pathwise-lepski}}

We first establish two bounds for the covariance estimates. The argument uses the local continuous/jump decomposition for volatility estimation. See Section~2.1.5 and Chapter~9 in Jacod and Protter~\cite{supp:JacodProtter2012} and Section~4 in Jacod and Rosenbaum~\cite{supp:JacodRosenbaum2013}. All estimates below are scalar or entrywise, with constants independent of the coordinate indices and Brownian dimension.

Write $I_i^n=((i-1)\Delta_n,i\Delta_n]$, $I_{n,k}=I_{n,k}(\tau)$, and $\E_{i-1}^n(\cdot)=\E(\cdot\mid\mathcal F_{(i-1)\Delta_n})$. Only increments in $I_n^{\max}$ are used. On that interval define
\[
A_t=\int_{\tau_n}^t b'_s\,ds, \qquad Z'_t=Z_{\tau_n}+A_t+Z_t^c-Z_{\tau_n}^c.
\]
Assumption~\ref{ass:jump} makes this local definition well defined. In particular,
\[
\Delta_i^nZ_a=\Delta_i^nZ'_a+ \int_{I_i^n}\int_{\mathcal E}\delta_a(t,z)\,\mu(dt,dz), \qquad i\in I_{n,k_{\max}}.
\]
Set
\[
\widehat\Sigma'(k)=\frac1{k\Delta_n} \sum_{i\in I_{n,k}}\Delta_i^n Z'(\Delta_i^n Z')^\top, \qquad \overline\Sigma(k)=\frac1{k\Delta_n} \int_{\tau_n}^{\tau_n+k\Delta_n}\Sigma_s\,ds.
\]
We use the decomposition
\begin{equation}
\widehat\Sigma^{\mathrm{tr}}(k)-\Sigma_\tau =\{\widehat\Sigma^{\mathrm{tr}}(k)-\widehat\Sigma'(k)\} +\{\widehat\Sigma'(k)-\overline\Sigma(k)\} +\{\overline\Sigma(k)-\Sigma_\tau\}. \label{eq:supp-covariance-decomposition}
\end{equation}

\begin{supplemma}\label{lem:supp-continuous}
Under Assumption~\ref{ass:jump}, for every deterministic $x\ge0$, with probability at least $1-Ce^{-x}$,
\begin{equation}
\|\widehat\Sigma'(k)-\overline\Sigma(k)\|_{\max} \le C\{a_{n,k}(x)+n^{-1/2}\},\qquad k\in\mathcal K_n. \label{eq:supp-continuous-covariance}
\end{equation}
If $a_{n,k_{\min}}(0)=O(1)$, then
\begin{equation}
\max_{k\in\mathcal K_n}\max_{a\le\bar p_n} \frac1{k\Delta_n}\sum_{i\in I_{n,k}}(\Delta_i^nZ'_a)^2 =O_{\Prb}(1). \label{eq:quadratic-average}
\end{equation}
\end{supplemma}

\begin{proof}
For a scalar continuous local martingale $M$ and a stopping time $T$, suppose that $\langle M\rangle_{T+u}-\langle M\rangle_T\le Ku$ for $0\le u\le h$. The nonnegative exponential local martingale, stopped at the first crossing of a positive level, gives
\[
\Prb\left(\sup_{0\le u\le h}(M_{T+u}-M_T)>y \mid\mathcal F_T\right) \le\exp\{-\lambda y+\lambda^2Kh/2\},\qquad \lambda>0.
\]
Taking $\lambda=y/(Kh)$ and applying the same argument to $-M$ yields
\begin{equation}
\Prb\left(\sup_{0\le u\le h}|M_{T+u}-M_T|>y \mid\mathcal F_T\right) \le2\exp\{-y^2/(2Kh)\}. \label{eq:max-exp}
\end{equation}
Assumption~\ref{ass:jump} permits this bound for each scalar coordinate of $Z^c$. Integrating the conditional tail, for every integer $q\ge2$,
\begin{align}
\E_{i-1}^n\left[ \sup_{s\in I_i^n}|Z^c_{a,s}-Z^c_{a,(i-1)\Delta_n}|^q\right] &\le 2q\int_0^\infty y^{q-1}e^{-y^2/(2K\Delta_n)}\,dy\notag\\ &\le (Cq)^{q/2}\Delta_n^{q/2}. \label{eq:bdg-increment}
\end{align}
For the corresponding fixed-order conditional estimates, see Eq.~(2.1.34) in Jacod and Protter~\cite{supp:JacodProtter2012}.

For $a,b\le\bar p_n$, let
\begin{equation}
Q^n_{i,ab}=\Delta_i^nZ_a^c\Delta_i^nZ_b^c -\int_{I_i^n}\Sigma_{ab,s}\,ds. \label{eq:Q-def}
\end{equation}
Integration by parts in \eqref{eq:Q-def} gives
\begin{align*}
Q^n_{i,ab} ={}&\int_{I_i^n}(Z^c_{a,s}-Z^c_{a,(i-1)\Delta_n})\,dZ^c_{b,s}\\ &+\int_{I_i^n}(Z^c_{b,s}-Z^c_{b,(i-1)\Delta_n})\,dZ^c_{a,s}.
\end{align*}
By Assumption~\ref{ass:jump},
\[
\E_{i-1}^n\int_{I_i^n}|Z^c_{a,s}-Z^c_{a,(i-1)\Delta_n}|^2\Sigma_{bb,s}\,ds\le K^2\Delta_n^2/2<\infty,\qquad \E_{i-1}^nQ^n_{i,ab}=0.
\]
Since $|\Sigma_{ab,s}|\le\sqrt{\Sigma_{aa,s}\Sigma_{bb,s}}\le K$, conditional Cauchy--Schwarz and \eqref{eq:bdg-increment} give
\begin{align*}
\E_{i-1}^n|\Delta_i^nZ_a^c\Delta_i^nZ_b^c|^q &\le\left\{\E_{i-1}^n|\Delta_i^nZ_a^c|^{2q} \E_{i-1}^n|\Delta_i^nZ_b^c|^{2q}\right\}^{1/2}\\ &\le(Cq)^q\Delta_n^q.
\end{align*}
By Stirling's inequality, for fixed $V_0,K_0>0$,
\[
\E_{i-1}^n|Q^n_{i,ab}|^q\le(Cq)^q\Delta_n^q\le\frac{q!}{2}V_0^2\Delta_n^2(K_0\Delta_n)^{q-2}.
\]
Expanding the conditional exponential gives
\[
\E_{i-1}^n e^{\lambda Q^n_{i,ab}}\le\exp\left\{\frac{\lambda^2V_0^2\Delta_n^2}{2(1-|\lambda|K_0\Delta_n)}\right\},\qquad |\lambda|K_0\Delta_n<1.
\]
Iterating conditional expectations over the $k$ increments and applying the exponential Markov inequality to both signs gives
\begin{equation}
\Prb\left( \left|\sum_{i\in I_{n,k}}Q^n_{i,ab}\right| >\sqrt{2kV_0^2\Delta_n^2t}+2K_0\Delta_n t \right)\le2e^{-t},\qquad t>0. \label{eq:bernstein}
\end{equation}
With $t=\Lambda_n(x)$, \eqref{eq:bernstein} gives
\[
\Prb\left(\exists k\in\mathcal K_n:\max_{a,b}\left|\frac1{k\Delta_n}\sum_{i\in I_{n,k}}Q^n_{i,ab}\right|>Ca_{n,k}(x)\right)\le2M_n\bar p_n^2e^{-\Lambda_n(x)}\le Ce^{-x}.
\]
On the complementary event,
\begin{equation}
\max_{a,b}\left|\frac1{k\Delta_n} \sum_{i\in I_{n,k}}Q^n_{i,ab}\right| \le C a_{n,k}(x),\qquad k\in\mathcal K_n. \label{eq:supp-Q-uniform}
\end{equation}
The diagonal case of \eqref{eq:supp-Q-uniform} implies
\[
\max_a\frac1{k\Delta_n}\sum_{i\in I_{n,k}}(\Delta_i^nZ_a^c)^2 \le K+C a_{n,k}(x).
\]
Since $|\Delta_i^nA_a|\le K\Delta_n$, expansion of $\Delta_i^nZ'=\Delta_i^nA+\Delta_i^nZ^c$ gives
\begin{align*}
\widehat\Sigma'_{ab}(k)-\overline\Sigma_{ab}(k) ={}&\frac1{k\Delta_n}\sum_{i\in I_{n,k}}Q^n_{i,ab}\\ &+\frac1{k\Delta_n}\sum_{i\in I_{n,k}}\Delta_i^nA_a\Delta_i^nZ_b^c +\frac1{k\Delta_n}\sum_{i\in I_{n,k}}\Delta_i^nA_b\Delta_i^nZ_a^c\\ &+\frac1{k\Delta_n}\sum_{i\in I_{n,k}}\Delta_i^nA_a\Delta_i^nA_b.
\end{align*}
The two cross terms obey
\begin{align*}
\max_{a,b}\frac1{k\Delta_n}\sum_{i\in I_{n,k}}|\Delta_i^nA_a\Delta_i^nZ_b^c|&\le K\sqrt{\Delta_n}\left\{\max_b\frac1{k\Delta_n}\sum_{i\in I_{n,k}}(\Delta_i^nZ_b^c)^2\right\}^{1/2}\\ &\le C\sqrt{\Delta_n}\{1+a_{n,k}(x)\}^{1/2},
\end{align*}
and
\[
\max_{a,b}\frac1{k\Delta_n}\sum_{i\in I_{n,k}}|\Delta_i^nA_a\Delta_i^nA_b|\le K^2\Delta_n.
\]
Using $\sqrt{\Delta_n a_{n,k}(x)}\le\{\Delta_n+a_{n,k}(x)\}/2$ proves \eqref{eq:supp-continuous-covariance}.

For \eqref{eq:quadratic-average}, apply the diagonal bound with an arbitrary $u\ge0$ in place of $x$. If $a_{n,k_{\min}}(0)=O(1)$, then, for all sufficiently large $n$,
\[
\Prb\left( \max_{k,a}\frac1{k\Delta_n}\sum_{i\in I_{n,k}}(\Delta_i^nZ'_a)^2 >C(1+\sqrt u+u)\right)\le Ce^{-u}.
\]
Indeed,
\[
a_{n,k_{\min}}(u)\le a_{n,k_{\min}}(0)+\sqrt{u/k_{\min}}+u/k_{\min},\qquad (\Delta_i^nZ'_a)^2\le2(\Delta_i^nZ_a^c)^2+2K^2\Delta_n^2.
\]
Thus the tail bound tends to zero as $u\to\infty$, proving \eqref{eq:quadratic-average}.
\end{proof}

\begin{supplemma}\label{lem:supp-truncation}
Suppose Assumption~\ref{ass:jump}, the scale condition~\eqref{eq:scale-bounds}, and conditions~\eqref{eq:threshold-range}, \eqref{eq:grid-compatibility}, and \eqref{eq:grid-condition} hold. Then
\begin{equation}
\max_{k\in\mathcal K_n} \|\widehat\Sigma^{\mathrm{tr}}(k)-\widehat\Sigma'(k)\|_{\max} =O_{\Prb}\bigl(M_n^{1/2}\vartheta_n^{1-r_0/2} +M_n\vartheta_n^{2-r_0}\bigr). \label{eq:supp-truncation-rate}
\end{equation}
In particular,
\begin{equation}
\max_{k\in\mathcal K_n} \frac{\|\widehat\Sigma^{\mathrm{tr}}(k)-\widehat\Sigma'(k)\|_{\max}} {a_{n,k}(x)+n^{-1/2}}=o_{\Prb}(1). \label{eq:relative-truncation}
\end{equation}
\end{supplemma}

\begin{proof}
Put
\[
\gamma(t,z)=\max_{a\le\bar p_n}(|\delta_a(t,z)|\wedge1), \qquad v=\alpha_0\underline s\,\vartheta_n.
\]
This threshold is deterministic. Assumption~\ref{ass:jump} implies, for $0<u<1$ and $q=1,2$,
\begin{equation}
\int_{\mathcal E}\gamma^q\mathbf1_{\{\gamma\le u\}}\,d\lambda \le Ku^{q-r_0}, \qquad \int_{\mathcal E}\mathbf1_{\{\gamma>u\}}\,d\lambda\le Ku^{-r_0}. \label{eq:gamma-power}
\end{equation}
Here and below the arguments $(t,z)$ are suppressed inside integrals when unambiguous. The convention in Assumption~\ref{ass:jump} covers $r_0=0$.

For all sufficiently large $n$, $|\Delta_i^nA_a|\le K\Delta_n\le v/8$. By \eqref{eq:max-exp} and a union bound,
\begin{equation}
\Prb\left(\max_{a\le\bar p_n}\max_{i\in I_{n,k_{\max}}} |\Delta_i^nZ'_a|>v/4\right) \le2\bar p_n k_{\max}e^{-cv^2/\Delta_n}=o(1). \label{eq:continuous-small}
\end{equation}
Here \eqref{eq:grid-compatibility} gives the last equality. Set
\[
\mathcal C_n=\left\{\max_{a\le\bar p_n}\max_{i\in I_{n,k_{\max}}}|\Delta_i^nZ'_a|\le v/4\right\}.
\]

For $i\in I_{n,k_{\max}}$, let
\[
N_i(v)=\int_{I_i^n}\int_{\mathcal E}\mathbf1_{\{\gamma>v/8\}}\,d\mu, \qquad S_i(v)=\int_{I_i^n}\int_{\mathcal E}\gamma\mathbf1_{\{\gamma\le v/8\}}\,d\mu,
\]
and set
\[
B_i(v)=\mathbf1_{\{N_i(v)>0\}\cup\{S_i(v)>v/4\}}.
\]
Their integrands are predictable. The compensator identities and \eqref{eq:gamma-power} give
\begin{align}
\E_{i-1}^nN_i(v)&\le C\Delta_n v^{-r_0},\label{eq:N-bound}\\ \E_{i-1}^nS_i(v)&\le C\Delta_n v^{1-r_0}.\label{eq:S-mean}
\end{align}
For the second moment, use the compensated-integral decomposition $S_i(v)=M_i(v)+A_i(v)$, where
\begin{align*}
M_i(v)&=\int_{I_i^n}\int_{\mathcal E}\gamma\mathbf1_{\{\gamma\le v/8\}}\,d(\mu-\nu),\\ A_i(v)&=\int_{I_i^n}\int_{\mathcal E}\gamma\mathbf1_{\{\gamma\le v/8\}}\,d\nu.
\end{align*}
For the corresponding square-integrability and finite-variation estimates, see Lemmas~2.1.5 and~2.1.7 in Jacod and Protter~\cite{supp:JacodProtter2012}. In the present scalar setting, conditional isometry and the pathwise compensator bound yield
\begin{align*}
\E_{i-1}^n M_i(v)^2 &=\E_{i-1}^n\int_{I_i^n}\int_{\mathcal E}\gamma^2 \mathbf1_{\{\gamma\le v/8\}}\,d\nu \le C\Delta_n v^{2-r_0},\\ 0\le A_i(v)&\le C\Delta_n v^{1-r_0}.
\end{align*}
Hence
\begin{equation}
\E_{i-1}^nS_i(v)^2 \le2\E_{i-1}^nM_i(v)^2+2\E_{i-1}^nA_i(v)^2 \le C\Delta_n v^{2-r_0}+C\Delta_n^2v^{2-2r_0}. \label{eq:S-second}
\end{equation}
By $B_i(v)\le N_i(v)+4S_i(v)/v$, \eqref{eq:N-bound}, and \eqref{eq:S-mean},
\begin{equation}
\E_{i-1}^nB_i(v)\le\E_{i-1}^nN_i(v)+\frac4v\E_{i-1}^nS_i(v)\le C\Delta_n v^{-r_0}. \label{eq:B-bound}
\end{equation}
These conditional estimates are taken with respect to the original filtration.

On $\mathcal H_n\cap\mathcal C_n$, $v\le\vartheta_{a,n}\le(\overline s/\underline s)v$, and
\[
B_i(v)=0\quad\Longrightarrow\quad N_i(v)=0,\quad |\Delta_i^nZ_a-\Delta_i^nZ'_a|\le S_i(v)\le v/4,\quad |\Delta_i^nZ_a|\le v/2\le\vartheta_{a,n}.
\]
Then $B_i(v)=0$ implies $\widetilde\Delta_i^nZ_a=\Delta_i^nZ_a$. Put
\[
D^n_{i,ab}=\widetilde\Delta_i^nZ_a\widetilde\Delta_i^nZ_b -\Delta_i^nZ'_a\Delta_i^nZ'_b.
\]
On $\{B_i(v)=0\}$ the product expansion gives
\begin{align*}
D^n_{i,ab} ={}&\Delta_i^nZ'_a(\Delta_i^nZ_b-\Delta_i^nZ'_b) +\Delta_i^nZ'_b(\Delta_i^nZ_a-\Delta_i^nZ'_a)\\ &+(\Delta_i^nZ_a-\Delta_i^nZ'_a) (\Delta_i^nZ_b-\Delta_i^nZ'_b).
\end{align*}
On $\{B_i(v)=1\}$,
\[
|\widetilde\Delta_i^nZ_a\widetilde\Delta_i^nZ_b|+|\Delta_i^nZ'_a\Delta_i^nZ'_b|\le Cv^2.
\]
Hence, uniformly in $a,b$,
\[
|D^n_{i,ab}| \le (|\Delta_i^nZ'_a|+|\Delta_i^nZ'_b|)S_i(v) +S_i(v)^2+Cv^2B_i(v).
\]
Summing first, and then applying Cauchy--Schwarz, gives
\begin{align}
\|\widehat\Sigma^{\mathrm{tr}}(k)-\widehat\Sigma'(k)\|_{\max} &\le C\left\{\max_a\frac1{k\Delta_n} \sum_{i\in I_{n,k}}(\Delta_i^nZ'_a)^2\right\}^{1/2} \left\{\frac1{k\Delta_n}\sum_{i\in I_{n,k}}S_i(v)^2\right\}^{1/2}\notag\\ &\quad+\frac C{k\Delta_n}\sum_{i\in I_{n,k}}S_i(v)^2 +\frac{Cv^2}{k\Delta_n}\sum_{i\in I_{n,k}}B_i(v). \label{eq:truncation-decomp}
\end{align}
By \eqref{eq:grid-condition} and Lemma~\ref{lem:supp-continuous},
\[
a_{n,k_{\min}}(0)\le c_0\quad\Longrightarrow\quad\max_{k\in\mathcal K_n}\max_a\frac1{k\Delta_n}\sum_{i\in I_{n,k}}(\Delta_i^nZ'_a)^2=O_{\Prb}(1).
\]
Nonnegativity and \eqref{eq:S-second}--\eqref{eq:B-bound} give
\begin{align*}
\E\left[\max_{k\in\mathcal K_n}\frac1{k\Delta_n} \sum_{i\in I_{n,k}}S_i(v)^2\right] &\le CM_n\{v^{2-r_0}+\Delta_n v^{2-2r_0}\},\\ \E\left[\max_{k\in\mathcal K_n}\frac{v^2}{k\Delta_n} \sum_{i\in I_{n,k}}B_i(v)\right] &\le CM_n v^{2-r_0}.
\end{align*}
Since $\Delta_n v^{-r_0}\asymp\Delta_n^{1-\varpi r_0}\to0$, Markov's inequality yields
\[
\max_k\frac1{k\Delta_n}\sum_{i\in I_{n,k}}S_i(v)^2 =O_{\Prb}(M_n\vartheta_n^{2-r_0}),
\]
\begin{equation}
\max_k\frac{v^2}{k\Delta_n}\sum_{i\in I_{n,k}}B_i(v) =O_{\Prb}(M_n\vartheta_n^{2-r_0}). \label{eq:B-sum}
\end{equation}
Since $\Prb(\mathcal H_n^c)+\Prb(\mathcal C_n^c)=o(1)$, \eqref{eq:truncation-decomp}, \eqref{eq:quadratic-average}, and \eqref{eq:B-sum} yield \eqref{eq:supp-truncation-rate}. Moreover, by \eqref{eq:grid-compatibility},
\[
\max_{k\in\mathcal K_n}\frac{\|\widehat\Sigma^{\mathrm{tr}}(k)-\widehat\Sigma'(k)\|_{\max}}{a_{n,k}(x)+n^{-1/2}}\le\frac{\max_k\|\widehat\Sigma^{\mathrm{tr}}(k)-\widehat\Sigma'(k)\|_{\max}}{a_{n,k_{\max}}(x)+n^{-1/2}}=O_{\Prb}\left(\frac{M_n^{1/2}\vartheta_n^{1-r_0/2}+M_n\vartheta_n^{2-r_0}}{a_{n,k_{\max}}(x)+n^{-1/2}}\right)=o_{\Prb}(1).
\]
\end{proof}

\begin{proof}[Proof of Theorem~\ref{thm:joint-adaptation}]
Intersect the event of Lemma~\ref{lem:supp-continuous} with $\mathcal G_n$ and with the event on which the maximum in \eqref{eq:relative-truncation} is at most one. Their intersection has probability at least $1-Ce^{-x}-o(1)$. On this event, \eqref{eq:supp-covariance-decomposition} and Assumption~\ref{ass:local-characteristics} give
\begin{equation}
\|\widehat\Sigma^{\mathrm{tr}}(k)-\Sigma_\tau\|_{\max} \le C\{a_{n,k}(x)+(k/n)^\alpha+n^{-1/2}\}, \qquad k\in\mathcal K_n. \label{eq:covariance-final}
\end{equation}
For sufficiently small $c_0$ in \eqref{eq:grid-condition},
\[
\min_{k\in\mathcal K_n}\min_a\widehat\Sigma^{\mathrm{tr}}_{aa}(k)\ge\min_a\Sigma_{aa,\tau}-\max_k\|\widehat\Sigma^{\mathrm{tr}}(k)-\Sigma_\tau\|_{\max}\ge c_--Cc_0\ge c_-/2>v_n.
\]

For positive-semidefinite $M$ with positive diagonal, let $\mathcal C(M)_{ab}=M_{ab}/\sqrt{M_{aa}M_{bb}}$. The function $g(x,y,z)=x(yz)^{-1/2}$ has derivatives
\[
\partial_xg=(yz)^{-1/2},\qquad \partial_yg=-\frac{x}{2y^{3/2}z^{1/2}},\qquad \partial_zg=-\frac{x}{2y^{1/2}z^{3/2}}.
\]
Along the segment joining $\Sigma_\tau$ and $\widehat\Sigma^{\mathrm{tr}}(k)$, the diagonal entries are bounded above and away from zero. The matrices on this segment are positive semidefinite, so their off-diagonal entries satisfy $|M_{ab}|\le\sqrt{M_{aa}M_{bb}}$. The mean-value theorem therefore gives the dimension-free bound
\begin{equation}
\|\mathcal C(\widehat\Sigma^{\mathrm{tr}}(k))- \mathcal C(\Sigma_\tau)\|_{\max} \le C\|\widehat\Sigma^{\mathrm{tr}}(k)-\Sigma_\tau\|_{\max}. \label{eq:corr-map}
\end{equation}
Combining \eqref{eq:corr-map} with \eqref{eq:covariance-final} and taking the covariate block and response column proves
\begin{equation}
\|\widehat r(k)-r_\tau\|_\infty\vee \|\widehat R(k)-R_\tau\|_{\max} \le C\{a_{n,k}(x)+(k/n)^\alpha+n^{-1/2}\}, \qquad k\in\mathcal K_n. \label{eq:correlation-final}
\end{equation}
\end{proof}

\begin{proof}[Proof of Corollary~\ref{cor:pathwise-lepski}]
We use Lepski's comparison argument \cite{supp:Lepski1991,supp:GoldenshlugerLepski2011}. On the event in \eqref{eq:correlation-final}, restriction to any set of columns gives
\begin{equation}
\|\widehat r(k)-r_\tau\|_\infty\vee \|\widehat R_{\cdot A}(k)-R_{\cdot A,\tau}\|_{\max} \le C_0\{a_{n,k}(x)+(k/n)^\alpha+n^{-1/2}\} \label{eq:TA-bound}
\end{equation}
for every $k\in\mathcal K_n$ and $|A|\le q_n$. Choose $k_0\in\mathcal K_n$ within a fixed factor of $k_n^\star=n^{2\alpha/(2\alpha+1)}\Lambda_n(x)^{1/(2\alpha+1)}$. The grid assumption and $k_{\max}/n\to0$ imply $\Lambda_n(x)=o(n)$. Since
\[
k\le k_n^\star\ \Longrightarrow\ \sqrt{\Lambda_n(x)/k}\ge\{\Lambda_n(x)/n\}^{\alpha/(2\alpha+1)},\qquad k\ge k_n^\star\ \Longrightarrow\ (k/n)^\alpha\ge\{\Lambda_n(x)/n\}^{\alpha/(2\alpha+1)},
\]
evaluation at $k_0$ gives
\[
a_{n,k_0}(x)\asymp\varepsilon_n(x), \qquad (k_0/n)^\alpha+n^{-1/2}\le C a_{n,k_0}(x).
\]
For every $\ell\le k_0$, \eqref{eq:TA-bound} and monotonicity of the radius imply
\begin{align*}
d_{k_0\ell}(A) &\le C_0\{a_{n,k_0}(x)+a_{n,\ell}(x) +(k_0/n)^\alpha+(\ell/n)^\alpha+2n^{-1/2}\}\\ &\le C_1\{a_{n,k_0}(x)+a_{n,\ell}(x)\}.
\end{align*}
For $C_{\mathrm L}\ge C_1$, \eqref{eq:lepski-selector} gives $\widehat k(A)\ge k_0$ and $a_{n,\widehat k(A)}(x)\le a_{n,k_0}(x)$. Hence
\begin{align*}
&\|\widehat r\{\widehat k(A)\}-r_\tau\|_\infty \vee\|\widehat R_{\cdot A}\{\widehat k(A)\}-R_{\cdot A,\tau}\|_{\max}\\ &\quad\le d_{\widehat k(A),k_0}(A) +C_0\{a_{n,k_0}(x)+(k_0/n)^\alpha+n^{-1/2}\}\\ &\quad\le C_{\mathrm L}\{a_{n,\widehat k(A)}(x)+a_{n,k_0}(x)\} +C a_{n,k_0}(x) \le C\varepsilon_n(x).
\end{align*}
All restrictions and comparisons hold on the same entrywise event, including for data-dependent active sets.
\end{proof}

\section{A factor-type sufficient condition for Assumption~\ref{ass:local-characteristics}}
\label{sec:supp-factor-condition}

The following lemma verifies the local regularity condition under a finite-dimensional representation of the covariance process.

\begin{supplemma}\label{prop:suff-ass2}
Suppose that, on events $\mathcal G_n$ with $\Prb(\mathcal G_n)\to1$, the augmented covariance process has the representation
\[
\Sigma_t=A_{0,n}+\sum_{\nu=1}^{d_\Sigma}g_{\nu,t}A_{\nu,n}, \qquad d_\Sigma\le d_0<\infty, \qquad \max_{\nu\le d_\Sigma}\|A_{\nu,n}\|_{\max}\le C_A.
\]
Suppose that, on $I_n^{\max}$,
\[
|g_{\nu,t}-g_{\nu,\tau}|\le L_g|t-\tau|^\alpha, \qquad \nu\le d_\Sigma,
\]
and the diagonal entries of $\Sigma_t$ satisfy fixed positive lower and upper bounds. Then Assumption~\ref{ass:local-characteristics} holds with $L_\Sigma=2^\alpha d_0C_AL_g$.
\end{supplemma}

\begin{proof}
On $\mathcal G_n$,
\[
|\Sigma_{ab,t}-\Sigma_{ab,\tau}| \le\sum_{\nu=1}^{d_\Sigma}|g_{\nu,t}-g_{\nu,\tau}| |(A_{\nu,n})_{ab}| \le d_0C_AL_g|t-\tau|^\alpha.
\]
For $h=k\Delta_n$,
\[
0\le\tau_n-\tau<\Delta_n\le h,\qquad \sup_{t\in[\tau_n,\tau_n+h]}|t-\tau|\le2h.
\]
Consequently,
\[
\left|\frac1h\int_{\tau_n}^{\tau_n+h}\Sigma_{ab,t}\,dt-\Sigma_{ab,\tau}\right| \le\frac{d_0C_AL_g}{h}\int_{\tau_n}^{\tau_n+h}|t-\tau|^\alpha\,dt \le2^\alpha d_0C_AL_gh^\alpha.
\]
Taking the maximum over $(a,b)$ proves the averaged regularity bound. The diagonal bounds give the remaining requirement.
\end{proof}

For example, let $dg_{\nu,t}=b_{\nu,t}\,dt+dM_{\nu,t}$, where $M_\nu$ is a continuous local martingale and $|b_{\nu,t}|+d\langle M_\nu\rangle_t/dt\le K$ on a fixed right neighborhood of $\tau$. For $\alpha<1/2$, fixed $L_g>0$, and sufficiently small $h$,
\[
\sup_{0<u\le h}\frac{|\int_\tau^{\tau+u}b_{\nu,t}\,dt|}{u^\alpha}\le Kh^{1-\alpha}\le L_g/2.
\]
Applying \eqref{eq:max-exp} on dyadic intervals of lengths $2^{-m}h$ gives
\[
\Prb\left(\sup_{0<u\le h}\frac{|g_{\nu,\tau+u}-g_{\nu,\tau}|}{u^\alpha}>L_g\right) \le2\sum_{m=0}^{\infty} \exp\{-cL_g^2h^{-(1-2\alpha)}2^{m(1-2\alpha)}\}\longrightarrow0.
\]
A union bound over at most $d_0$ factors and $h=(k_{\max}+1)\Delta_n$ gives the stated local condition. The diagonal bounds are imposed separately as in the lemma.

\section{Proof of Theorem~\ref{thm:marginal-screening}}

\begin{proof}
If $S_\tau^m=\varnothing$, the inclusion is immediate. Otherwise, Corollary~\ref{cor:pathwise-lepski} and the signal condition imply that, with probability at least $1-Ce^{-cx}-o(1)$,

\[
\bnorm{\widehat r(\widehat k(\varnothing))-r_\tau}_\infty\le a_m/4.
\]

On this event,

\[
|\widehat r_j(\widehat k(\varnothing))|\ge3a_m/4 \quad(j\in S_\tau^m),\qquad |\widehat r_j(\widehat k(\varnothing))|\le a_m/4 \quad(j\notin S_\tau^m).
\]

Therefore,
\[
d_n\ge|S_\tau^m|\ \Longrightarrow\ S_\tau^m\subseteq\widehat S_m(d_n),\qquad d_n=|S_\tau^m|\ \Longrightarrow\ \widehat S_m(d_n)=S_\tau^m.
\]

For the regression-support consequence stated after the theorem, work on the event

\[
\bnorm{\widehat r(\widehat k(\varnothing))-r_\tau}_\infty\le a/4.
\]

For $j\in S_\tau^\beta$ and $|\widehat r_\ell(\widehat k(\varnothing))|\ge|\widehat r_j(\widehat k(\varnothing))|$,

\[
|r_{\ell,\tau}| \ge|\widehat r_\ell(\widehat k(\varnothing))|-a/4 \ge a/2.
\]

The set $H=\{\ell:|r_{\ell,\tau}|\ge a/2\}$ obeys

\[
|H|\,a^2/4\le\sum_{\ell\in H}r_{\ell,\tau}^2\le \bnorm{r_\tau}_2^2.
\]

Consequently, for $j\in S_\tau^\beta$,
\[
\left|\left\{\ell:|\widehat r_\ell(\widehat k(\varnothing))|\ge|\widehat r_j(\widehat k(\varnothing))|\right\}\right|\le|H|\le4\bnorm{r_\tau}_2^2/a^2\le d_n,
\]
and hence $S_\tau^\beta\subseteq\widehat S_m(d_n)$. Corollary~\ref{cor:pathwise-lepski} gives the stated probability.
\end{proof}

\section{Proof of Proposition~\ref{prop:partial-stability}}

\begin{proof}
All population correlation quantities in this proof are evaluated at $\tau$, and the subscript $\tau$ is suppressed. The active set $A$ is fixed. If $A=\varnothing$, then
\[
|\widehat\zeta_{j\mid\varnothing}-\zeta_{j\mid\varnothing}| =|\widehat r_j-r_j|\le\eta_r,
\]
which is covered by the stated bound after taking $C_{\mathrm{pr}}(\phi_-)\ge1$.

Now let $A\ne\varnothing$ and define
\[
E_A=\widehat R_{AA}-R_{AA}, \qquad e_A=\widehat r_A-r_A, \qquad u_A=R_{AA}^{-1}R_{Aj}, \qquad \widehat v_A=\widehat R_{AA}^{-1}\widehat r_A.
\]
For every unit vector $v\in\mathbb R^{|A|}$,
\[
|v^\top E_Av|\le\bnorm{E_A}_{\max}\bnorm{v}_1^2 \le |A|\bnorm{E_A}_{\max}.
\]
Since $E_A$ is symmetric,
\[
\bnorm{E_A}_{\op}=\sup_{\|v\|_2=1}|v^\top E_Av|\le q\eta_R\le\phi_-/2,\qquad \lambda_{\min}(\widehat R_{AA})\ge\phi_--\bnorm{E_A}_{\op}\ge\phi_-/2>0.
\]
Thus $\widehat R_{AA}$ is invertible.

We next record two Schur-complement bounds. If
\[
\begin{pmatrix}M&b\\ b^\top&c\end{pmatrix}\succeq0, \qquad M\succ0,
\]
then
\[
0\le\begin{pmatrix}-M^{-1}b\\1\end{pmatrix}^{\!\top}\begin{pmatrix}M&b\\b^\top&c\end{pmatrix}\begin{pmatrix}-M^{-1}b\\1\end{pmatrix}=c-b^\top M^{-1}b.
\]
If also $M\succeq\phi I$, then
\[
\bnorm{M^{-1}b}_2^2 \le\phi^{-1}b^\top M^{-1}b \le c/\phi.
\]
Apply these inequalities first to
\[
\begin{pmatrix}R_{AA}&R_{Aj}\\ R_{jA}&1\end{pmatrix}
\]
and then to the estimated active-response block. The population and estimated augmented correlation matrices are positive semidefinite, and the estimated response diagonal is at most one. Therefore
\[
\bnorm{u_A}_2\le\phi_-^{-1/2}, \qquad \bnorm{\widehat v_A}_2\le(2/\phi_-)^{1/2}.
\]
Since $|A|\le q$,
\[
\bnorm{u_A}_1\le\sqrt q\,\phi_-^{-1/2}, \qquad \bnorm{\widehat v_A}_1\le\sqrt{2q/\phi_-}.
\]

The residual-score difference is
\[
\widehat\zeta_{j\mid A}-\zeta_{j\mid A} =(\widehat r_j-r_j) -(\widehat R_{jA}-R_{jA})\widehat v_A -R_{jA}\left(\widehat v_A-R_{AA}^{-1}r_A\right).
\]
Because $\widehat R_{AA}\widehat v_A=\widehat r_A$,
\[
R_{AA}\left(\widehat v_A-R_{AA}^{-1}r_A\right) =-E_A\widehat v_A+e_A.
\]
Since $R_{jA}=u_A^\top R_{AA}$,
\[
R_{jA}\left(\widehat v_A-R_{AA}^{-1}r_A\right) =-u_A^\top E_A\widehat v_A+u_A^\top e_A.
\]
The four terms satisfy
\begin{align*}
|\widehat r_j-r_j|&\le\eta_r,\\ |(\widehat R_{jA}-R_{jA})\widehat v_A|&\le\bnorm{\widehat R_{jA}-R_{jA}}_\infty\bnorm{\widehat v_A}_1\le\eta_R\bnorm{\widehat v_A}_1,\\ |u_A^\top e_A|&\le\bnorm{u_A}_1\bnorm{e_A}_\infty\le\bnorm{u_A}_1\eta_r,\\ |u_A^\top E_A\widehat v_A|&\le\sum_{r,s\in A}|u_{A,r}|\,|(E_A)_{rs}|\,|\widehat v_{A,s}|\le\bnorm{u_A}_1\bnorm{E_A}_{\max}\bnorm{\widehat v_A}_1.
\end{align*}
Substituting the $\ell_1$ bounds gives
\begin{align*}
|\widehat\zeta_{j\mid A}-\zeta_{j\mid A}| &\le \eta_r+\sqrt{2q/\phi_-}\,\eta_R +\sqrt q\,\phi_-^{-1/2}\eta_r +\sqrt2\,q\phi_-^{-1}\eta_R\\ &\le C_{\mathrm{pr}}(\phi_-) \{\sqrt q\,\eta_r+q\eta_R\}.
\end{align*}
The bound is uniform in $j\notin A$, which proves Proposition~\ref{prop:partial-stability}.
\end{proof}

\section{Proof of Theorem~\ref{thm:pr-sisis}}

\begin{proof}
Let $\mathcal E_n$ be the event in Corollary~\ref{cor:pathwise-lepski}. Work on $\mathcal E_n\cap G_n^\beta$. With $C_0$ denoting the constant in that corollary and $C_0c_{\mathrm{stab}}\le1/2$, for $0<|A|\le q_n$,
\begin{align*}
\bigl\|\widehat R_{AA}\{\widehat k(A)\}-R_{AA,\tau}\bigr\|_{\mathrm{op}} &\le |A|C_0\varepsilon_n(x)\le\phi_-/2,\\ \lambda_{\min}\!\left[\widehat R_{AA}\{\widehat k(A)\}\right] &\ge\phi_-/2.
\end{align*}
Proposition~\ref{prop:partial-stability} gives
\begin{equation}
\max_{\substack{|A|\le q_n\\j\notin A}} \left|\widehat\zeta_{j\mid A}\{\widehat k(A)\}-\zeta_{j\mid A}\right| \le\Delta_{n,q_n}(x) \le\phi_-\underline a_{\beta,n}/4. \label{eq:uniform-residual-sample}
\end{equation}
These bounds include the final active block used for the optional ordering. In particular,
\[
|A^{(u)}|=\min\{um_n,q_n\},\qquad 0\le u\le L_n.
\]
If $S=\varnothing$, the assertion follows. Otherwise, all population quantities below are evaluated at $\tau$, and this subscript is suppressed. Put $s=|S|$ and $\Delta=\Delta_{n,q_n}(x)$.

For $U\subseteq\{1,\ldots,p_n\}$ with $|U|\le q_n+\bar s_n$, define
\[
F(U)=r_U^\top R_{UU}^{-1}r_U,\qquad F(\varnothing)=0,\qquad G(U)=F(S)-F(U).
\]
Since $r=R\theta$ and $\theta_{S^c}=0$,
\begin{align*}
F(S)&=\theta^\top R\theta\le\overline F_{\beta,n},\\ F(U)&=\max_z\{2z^\top r_U-z^\top R_{UU}z\},\\ G(U)&=\min_{\operatorname{supp}(z)\subseteq U} (\theta-z)^\top R(\theta-z)\ge0.
\end{align*}
Thus $F$ is nondecreasing and $G$ is nonincreasing under set inclusion on this domain. For disjoint $A,D$ with $|A\cup D|\le q_n+\bar s_n$, write
\[
H_{D\mid A}=R_{DD}-R_{DA}R_{AA}^{-1}R_{AD}.
\]
Terms involving $R_{AA}^{-1}$ are absent when $A=\varnothing$. The sparse eigenvalue condition yields
\[
y^\top H_{D\mid A}y =\min_z \begin{pmatrix}z\\y\end{pmatrix}^{\!\top} R_{A\cup D,A\cup D} \begin{pmatrix}z\\y\end{pmatrix} \ge\phi_-\|y\|_2^2.
\]
For $|A|\le q_n$ and $B=S\setminus A$,
\begin{align*}
\zeta_{B\mid A}&=H_{B\mid A}\theta_B,\\ G(A)&=\theta_B^\top H_{B\mid A}\theta_B \ge\phi_-\|\theta_B\|_2^2.
\end{align*}
Consequently, if $B\ne\varnothing$,
\begin{equation}
\max_{j\in B}|\zeta_{j\mid A}| \ge\frac{\|H_{B\mid A}\theta_B\|_2}{\sqrt{|B|}} \ge\frac{\phi_-\|\theta_B\|_2}{\sqrt{|B|}} \ge\phi_-a_\beta. \label{eq:population-residual-signal}
\end{equation}
For $j\notin A$ and $|A|<q_n$, completing the square gives
\begin{align*}
F(A\cup\{j\})-F(A) &=\max_{z\in\mathbb R} \left\{2z\zeta_{j\mid A} -z^2\bigl(1-R_{jA}R_{AA}^{-1}R_{Aj}\bigr)\right\}\\ &=\frac{\zeta_{j\mid A}^{\,2}} {1-R_{jA}R_{AA}^{-1}R_{Aj}} \ge\zeta_{j\mid A}^{\,2},
\end{align*}
where
\[
\phi_-\le1-R_{jA}R_{AA}^{-1}R_{Aj}\le1.
\]

At an update from $A$, let
\[
M(A)=\max_{j\notin A}|\zeta_{j\mid A}|,
\]
and let $j_A$ be first in the sample absolute-score ordering. If $S\nsubseteq A$, \eqref{eq:uniform-residual-sample} and \eqref{eq:population-residual-signal} imply
\begin{align*}
M(A)&\ge\phi_-a_\beta\ge4\Delta,\\ |\zeta_{j_A\mid A}| &\ge\max_{j\notin A} |\widehat\zeta_{j\mid A}\{\widehat k(A)\}|-\Delta\\ &\ge M(A)-2\Delta\ge M(A)/2.
\end{align*}
The next set $A^+$ contains $A\cup\{j_A\}$, so
\begin{equation}
F(A^+)-F(A)\ge\frac{M(A)^2}{4} \ge\frac{\phi_-^2a_\beta^2}{4}. \label{eq:pr-greedy-increase}
\end{equation}
This also holds for the last, possibly smaller, block. If $L_n\ge\lceil4\overline F_{\beta,n}/ (\phi_-^2\underline a_{\beta,n}^{\,2})\rceil$, then
\[
S\nsubseteq A^{(L_n)} \quad\Longrightarrow\quad F(S)>F(A^{(L_n)}) \ge\frac{L_n\phi_-^2a_\beta^2}{4} \ge\overline F_{\beta,n}\ge F(S),
\]
a contradiction.

For the sparsity bound, we adapt the grouping argument from Lemmas~A.4--A.5 in Zhang~\cite{supp:Zhang2011OMP}. Fix
\[
K_S\ge\max\{1,\lambda_{\max}(R_{SS})\},\qquad \chi=\frac{4K_S}{\phi_-},\qquad c_{\mathrm{gr}}=\frac4{\phi_-}\log(2\chi).
\]
The unit diagonal gives $0<\phi_-\le1$, so $\chi\ge4$ and $c_{\mathrm{gr}}\ge1$. We prove by induction on $b=|S\setminus A^{(u)}|$ that, for every path index $u$,
\begin{equation}
L_n-u\ge\lceil4c_{\mathrm{gr}}b\rceil \quad\Longrightarrow\quad S\subseteq A^{(u+\lceil4c_{\mathrm{gr}}b\rceil)}. \label{eq:pr-remaining-budget}
\end{equation}
The case $b=0$ is immediate.

First, let $A_0$ be a set on the path, $\varnothing\ne T\subseteq S\setminus A_0$, and $G(A_0\cup T)\le v$. For a later update set $A\supseteq A_0$, set $D=T\setminus A$. Then
\[
G(A\cup T)\le v,\qquad |A\cup T|\le|A\cup S|\le q_n+\bar s_n.
\]
If $D\ne\varnothing$, the Schur complement identity gives
\begin{align*}
\{G(A)-v\}_+ &\le F(A\cup T)-F(A)\\ &=\zeta_{D\mid A}^\top H_{D\mid A}^{-1}\zeta_{D\mid A}\\ &\le\phi_-^{-1}\|\zeta_{D\mid A}\|_2^2 \le\frac{|T|}{\phi_-}M(A)^2.
\end{align*}
If $D=\varnothing$, then $\{G(A)-v\}_+=0$. If $G(A)\le v$, then $G(A^+)\le G(A)\le v$. If $G(A)>v\ge0$, then $S\nsubseteq A$, and \eqref{eq:pr-greedy-increase} applies. Hence
\begin{equation}
\{G(A^+)-v\}_+ \le\left(1-\frac{\phi_-}{4|T|}\right) \{G(A)-v\}_+. \label{eq:pr-excess-contraction}
\end{equation}

Now take $A_0=A^{(u)}$ with $b\ge1$, and order its missing true coordinates as
\[
S\setminus A_0=\{j_1,\ldots,j_b\},\qquad |\theta_{j_1}|\ge\cdots\ge|\theta_{j_b}|>0.
\]
For $\ell\ge0$, put
\[
E_\ell=\sum_{i=2^\ell}^{b}\theta_{j_i}^2,\qquad T_\ell=\{j_1,\ldots,j_{\min(b,2^\ell-1)}\},
\]
with empty sums and sets interpreted as zero and $\varnothing$, respectively. Define
\[
J=\min\{\ell\ge1:E_{\ell-1}\ge\chi E_\ell\},\qquad r_*=2^{J-1}.
\]
Then
\begin{gather*}
J\le\lfloor\log_2b\rfloor+1,\qquad 1\le r_*\le b,\qquad E_{J-1}>0,\\ E_{\ell-1}<\chi E_\ell\quad(1\le\ell<J),\qquad E_J\le\chi^{-1}E_{J-1}.
\end{gather*}
For $1\le\ell\le J$, set
\[
t_\ell=\lceil c_{\mathrm{gr}}|T_\ell|\rceil,\qquad u_0=u,\qquad u_\ell=u+\sum_{h=1}^{\ell}t_h.
\]
The update indices satisfy
\begin{align*}
u_J-u &\le c_{\mathrm{gr}}\sum_{\ell=1}^{J}|T_\ell|+J\\ &\le c_{\mathrm{gr}}\sum_{\ell=1}^{J}(2^\ell-1)+J\\ &=c_{\mathrm{gr}}(4r_*-2-J)+J \le4c_{\mathrm{gr}}r_*\le4c_{\mathrm{gr}}b,\\ u_J&\le u+\lceil4c_{\mathrm{gr}}b\rceil\le L_n.
\end{align*}
The variational expression for $G$ gives
\[
G(A_0\cup T_\ell) \le K_S\|\theta_{S\setminus(A_0\cup T_\ell)}\|_2^2 =K_SE_\ell,\qquad 0\le\ell\le J.
\]
Applying \eqref{eq:pr-excess-contraction} over each group of $t_\ell$ updates,
\begin{align*}
G(A^{(u_\ell)}) &\le K_SE_\ell+ \left(1-\frac{\phi_-}{4|T_\ell|}\right)^{t_\ell} \{G(A^{(u_{\ell-1})})-K_SE_\ell\}_+\\ &\le K_SE_\ell+(2\chi)^{-1}G(A^{(u_{\ell-1})}), \qquad 1\le\ell\le J,
\end{align*}
since
\[
\left(1-\frac{\phi_-}{4|T_\ell|}\right)^{t_\ell} \le\exp\!\left(-\frac{\phi_-t_\ell}{4|T_\ell|}\right) \le(2\chi)^{-1}.
\]
Starting from $G(A^{(u_0)})\le K_SE_0$,
\begin{align*}
G(A^{(u_1)})&\le K_S(E_1+\chi^{-1}E_0),\\ G(A^{(u_\ell)}) &\le K_SE_\ell+ \frac{K_S}{2\chi}(E_{\ell-1}+\chi^{-1}E_{\ell-2})\\ &\le K_S(E_\ell+\chi^{-1}E_{\ell-1}), \qquad 2\le\ell\le J.
\end{align*}
Thus
\begin{equation}
G(A^{(u_J)}) \le K_S(E_J+\chi^{-1}E_{J-1}) \le2\chi^{-1}K_SE_{J-1} =\frac{\phi_-}{2}E_{J-1}. \label{eq:pr-dyadic-gap}
\end{equation}
If $|(A^{(u_J)}\setminus A_0)\cap S|<r_*$, the ordering implies
\[
\|\theta_{S\setminus A^{(u_J)}}\|_2^2 \ge\sum_{i=r_*}^{b}\theta_{j_i}^2=E_{J-1}.
\]
Together with \eqref{eq:pr-dyadic-gap}, this would give
\[
0<\phi_-E_{J-1} \le G(A^{(u_J)}) \le\frac{\phi_-}{2}E_{J-1},
\]
a contradiction. Hence $b'=|S\setminus A^{(u_J)}|\le b-r_*$. Since $u_J-u$ is an integer,
\begin{align*}
(u_J-u)+\lceil4c_{\mathrm{gr}}b'\rceil &=\lceil u_J-u+4c_{\mathrm{gr}}b'\rceil\\ &\le\lceil4c_{\mathrm{gr}}b\rceil.
\end{align*}
Therefore $L_n-u_J\ge\lceil4c_{\mathrm{gr}}b'\rceil$. The induction hypothesis and nesting give
\[
S\subseteq A^{(u_J+\lceil4c_{\mathrm{gr}}b'\rceil)}\subseteq A^{(u+\lceil4c_{\mathrm{gr}}b\rceil)},
\]
which proves \eqref{eq:pr-remaining-budget}.

Taking $u=0$ in \eqref{eq:pr-remaining-budget},
\[
L_n\ge\left\lceil\frac{16s}{\phi_-} \log\!\left(\frac{8K_S}{\phi_-}\right)\right\rceil \quad\Longrightarrow\quad S\subseteq A^{(L_n)}.
\]
Since
\[
1\le\lambda_{\max}(R_{SS})\le\operatorname{tr}(R_{SS})=s\le\bar s_n,
\]
we may take $K_S=\bar s_n$. Then
\[
\left\lceil\frac{16s}{\phi_-} \log\!\left(\frac{8\bar s_n}{\phi_-}\right)\right\rceil \le \left\lceil\frac{16\bar s_n}{\phi_-} \log\!\left(1+\frac{8\bar s_n}{\phi_-}\right)\right\rceil.
\]
The minimum in \eqref{eq:block-size} therefore suffices, and
\begin{align*}
\Pp\{S\nsubseteq\widehat S_{\mathrm{PR}}\} &\le\Pp(\mathcal E_n^c)+\Pp\{(G_n^\beta)^c\}\\ &\le Ce^{-cx}+o(1).
\end{align*}
The initial eigenvalue bound proves the nonsingularity assertion.

Finally, on $\{\|R_\tau\|_{\mathrm{op}}\le\phi_+\}$,
\[
\lambda_{\max}(R_{SS})\le\phi_+, \qquad \left\lceil\frac{16s}{\phi_-} \log\!\left(\frac{8\phi_+}{\phi_-}\right)\right\rceil \le \left\lceil\frac{16\bar s_n}{\phi_-} \log\!\left(1+\frac{8\phi_+}{\phi_-}\right)\right\rceil.
\]
Taking $K_S=\phi_+$ and adding $\Pp\{\|R_\tau\|_{\mathrm{op}}>\phi_+\}=o(1)$ gives the stated alternative.
\end{proof}

\section{Proof of Theorem~\ref{thm:list-lower-bound}}
\label{sec:supp-proof-lower}

\begin{proof}
The triangular profile satisfies, for $s,t\in\mathbb R$,
\[
|f_a(t)-f_a(s)| \le a\min\{1,|t-s|/h_a\} \le a h_a^{-\alpha}|t-s|^\alpha =\frac{L_\Sigma}{2^\alpha}|t-s|^\alpha.
\]
For $h=k\Delta_n$, $0\le\tau_n-\tau<\Delta_n\le h$, and hence
\[
\frac1h\int_{\tau_n}^{\tau_n+h}|f_a(t)-f_a(\tau)|\,dt \le\frac{L_\Sigma}{2^\alpha}(2h)^\alpha=L_\Sigma h^\alpha.
\]
Under $P_{j,a}^{(n)}$, the covariance matrix is
\[
\Sigma_t= \begin{pmatrix}I_{p_n}&f_a(t)e_j\\f_a(t)e_j^\top&1\end{pmatrix}.
\]
In particular,
\[
\Sigma_{aa,t}=1,\qquad b_t=b'_t=0,\qquad \delta(t,z)=0,\qquad \theta_\tau=r_\tau=ae_j,\qquad S_\tau^\beta=S_\tau^m=\{j\}.
\]
Thus Assumptions~\ref{ass:jump} and~\ref{ass:local-characteristics} hold with $\mathcal G_n=\Omega$ and fixed diagonal bounds containing one.

We use the bounded-list form of Fano's inequality. See Eq.~(221) in Proposition~5 of Sason~\cite{supp:Sason2019}. Let $J$ be uniform on $\{1,\ldots,M\}$, let $V\mid J=j$ have law $P_j$, and let $\widehat L(V)$ have size at most $d<M$. Write $H$ for entropy, $I$ for mutual information, and $\operatorname{KL}(P,Q)=\int\log(dP/dQ)\,dP$. With $E=\mathbf1_{\{J\notin\widehat L(V)\}}$ and $P_e=\Prb(E=1)$,
\[
H(J\mid V)=H(E\mid V)+H(J\mid E,V) \le\log2+(1-P_e)\log d+P_e\log M.
\]
Since $H(J)=\log M=I(J;V)+H(J\mid V)$,
\begin{equation}
P_e\ge1-\frac{I(J;V)+\log2}{\log(M/d)}. \label{eq:list-fano}
\end{equation}
For $\overline P=M^{-1}\sum_jP_j$ and any $Q$ dominating these laws,
\[
\frac1M\sum_{j=1}^M\operatorname{KL}(P_j,Q) =I(J;V)+\operatorname{KL}(\overline P,Q),
\]
so
\begin{equation}
I(J;V)\le\frac1M\sum_{j=1}^M\operatorname{KL}(P_j,Q). \label{eq:mutual-kl}
\end{equation}

Let $P_0^{(n)}$ be the law under independent standard Brownian motions $X$ and $Y$ with the same zero initial values. The sampled levels and the increment array determine each other, so they have the same divergence. Under the alternatives, the integrands are deterministic and disjoint increments are independent Gaussian vectors. The coordinates outside $(X_j,Y)$ are independent of this pair and have the null law. On $I_i^n$,
\begin{align*}
\operatorname{Var}(\Delta_i^nX_j)&=\Delta_n,\\ \operatorname{Var}(\Delta_i^nY) &=\int_{I_i^n}\{f_a(t)^2+1-f_a(t)^2\}\,dt=\Delta_n,\\ \operatorname{Cov}(\Delta_i^nX_j,\Delta_i^nY) &=\int_{I_i^n}f_a(t)\,dt.
\end{align*}
After division by $\sqrt{\Delta_n}$, the covariance of the pair is
\[
\begin{pmatrix}1&\rho_i\\\rho_i&1\end{pmatrix}, \qquad \rho_i=\Delta_n^{-1}\int_{I_i^n}f_a(t)\,dt.
\]
With $|\rho_i|\le a\le1/2$, Gaussian divergence and Jensen's inequality give
\[
\operatorname{KL}(P_{j,a}^{(n)},P_0^{(n)})=-\frac12\sum_{i=1}^n\log(1-\rho_i^2)\le\sum_{i=1}^n\rho_i^2\le n\int_0^1 f_a(t)^2\,dt.
\]
The support of the bump is inside $(0,1)$, so
\[
\int_0^1 f_a(t)^2\,dt =2a^2\int_0^{h_a}(1-u/h_a)^2\,du =\frac23a^2h_a =\frac43L_\Sigma^{-1/\alpha}a^{(2\alpha+1)/\alpha}.
\]
Consequently,
\begin{equation}
\operatorname{KL}(P_{j,a}^{(n)},P_0^{(n)}) \le\frac43 nL_\Sigma^{-1/\alpha}a^{(2\alpha+1)/\alpha}. \label{eq:lower-kl}
\end{equation}

Apply \eqref{eq:list-fano}, \eqref{eq:mutual-kl}, and \eqref{eq:lower-kl} with $M=p_n$, $d=d_n$, $V=\mathcal O_n$, and $Q=P_0^{(n)}$. For every admissible list $\widehat S$,
\begin{equation}
\frac1{p_n}\sum_{j=1}^{p_n}P_{j,a_n}^{(n)}\{j\notin\widehat S\} \ge1-\frac{\frac43nL_\Sigma^{-1/\alpha}a_n^{(2\alpha+1)/\alpha}+\log2} {\log(p_n/d_n)}. \label{eq:lower-risk}
\end{equation}
Choose $c_1$ with $\frac43L_\Sigma^{-1/\alpha}c_1^{(2\alpha+1)/\alpha}\le1/4$. Then
\[
\frac{\frac43nL_\Sigma^{-1/\alpha}a_n^{(2\alpha+1)/\alpha}}{\log(p_n/d_n)}\le\frac14,\qquad \frac{\log2}{\log(p_n/d_n)}\le\frac12,
\]
and \eqref{eq:lower-risk} yields
\[
\sup_{P\in\mathcal P_{\alpha,n}(a_n)}P\{S_\tau^\beta\nsubseteq\widehat S\}\ge\frac1{p_n}\sum_{j=1}^{p_n}P_{j,a_n}^{(n)}\{j\notin\widehat S\}\ge1-\frac14-\frac12=\frac14.
\]
Under $P_{j,a_n}^{(n)}$,
\[
\{j\notin\widehat S\} =\{S_\tau^\beta\nsubseteq\widehat S\} =\{S_\tau^m\nsubseteq\widehat S\}.
\]
Taking the infimum over all admissible lists proves both assertions.
\end{proof}

\section{Computational details}

Truncated increments and cumulative sums are formed over the longest candidate window. Marginal covariances and diagonal covariances are stored at the $M_n$ candidate endpoints. At a PR-SISIS update with active set $A$, only the columns $\widehat R_{\cdot A}(k)$ required by the partial-residual scores are formed. A column is computed when its variable first enters the active set and is then cached.

With model cap $q_n$, at most $q_n$ distinct active columns are requested. Thus at most $O(p_nq_n)$ distinct covariate pairs are evaluated. Storing their values at all candidate windows requires $O(M_np_nq_n)$ window-specific correlation summaries, in addition to the $O(M_np_n)$ marginal and diagonal summaries and the $O(k_{\max}p_n)$ local increment block. A direct cumulative-sum implementation has arithmetic cost
\[
O\{k_{\max}p_n(1+q_n)\}
\]
for the marginal, diagonal, and requested-column products. The active-set Lepski comparisons add at most
\[
O\left(M_n^2p_n\sum_s\{1+|A^{(s)}|\}\right),
\]
and the active solves involve matrices of dimension at most $q_n$. For fixed $q_n$ and a fixed candidate grid, these costs are linear in $p_n$.

The active correlation block is checked before each solve, including the solve used for the final ordering. If it is singular, the current retained set is returned without a complete ordering. Otherwise, the partial-residual score uses the ordinary matrix inverse.

\section{Simulation design details}

\subsection{Reference designs}
\label{subsec:supp-simulation-dgp}

The reference configurations and signal grids are specified in Section~\ref{subsec:reference-designs} of the main paper. For Figure~\ref{fig:signal-strength}, each coefficient direction is scaled across its stated signal grid while the residual variance remains fixed. Table~\ref{tab:simulation-reference} uses $a_\beta=0.40$ for LC-R, HC-R, and HC-S, and the stated LW coefficient vector with $a_\beta\approx0.283$.

The reference designs use the same stochastic volatility, leverage, intraday periodicity, and finite-activity jumps. Let
\[
v_0=1, \qquad s(t)=1+0.15\sin(2\pi t), \qquad g(t)=1+0.10\sin\{2\pi(t-\tau)\}.
\]
The processes satisfy
\[
dv_t =\frac{6}{252}(1-v_t)\,dt +\frac{0.50}{\sqrt{252}}\sqrt{v_t}\,dU_t +dJ_t^v,
\]
\[
dX_t=s(t)\sqrt{v_t}\,R_t^{1/2}dB_t+dJ_t^X,
\]
and
\[
dY_t =\{g(t)\beta_0\}^\top dX_t +s(t)\sqrt{c_\varepsilon v_t}\,dW_t^\varepsilon +dJ_t^\varepsilon.
\]
Here $R_t^{1/2}$ is the lower triangular covariance factor. The standardized covariate innovations have covariance $R_t$ and covariance vector $(-0.30,0,\ldots,0)^\top$ with the volatility innovation. The residual Brownian motion is independent of these innovations.

The volatility process is simulated by full-truncation Euler with five latent substeps per observed increment and a floor of $10^{-10}$.

Independent Poisson approximations on the latent grid generate common, coordinate-specific covariate, and residual jump events with intensities $4/252$, $8/252$, and $8/252$. At a common event, covariate $j$ receives
\[
12\sqrt{v_{t-}/390}\,\xi_j,
\]
the residual receives
\[
12\sqrt{c_\varepsilon v_{t-}/390}\,\xi_\varepsilon,
\]
and volatility is multiplied by $\exp(\eta)$, where the $\xi$ variables are independent standard normal variables and
\[
\eta\sim N(-0.03,0.10^2).
\]
A coordinate-specific event applies the same covariate jump distribution to one uniformly selected coordinate. A residual event affects only the regression residual. The response inherits covariate jumps through $\{g(t)\beta_0\}^\top\Delta X_t$.

The candidate grid is
\[
\mathcal K_n=\{32,40,50,63,79,99,124,155\},
\]
with $C_L=0.40$ and $x=0$. For coordinate $a$, the full-day bipower scale and threshold are
\[
\widehat s_{a,n} =\left[ \max\left\{ \frac{\pi}{2}\sum_{i=2}^{n} |\Delta_i^nZ_a\Delta_{i-1}^nZ_a|, 10^{-12} \right\} \right]^{1/2}, \qquad \widehat\vartheta_{a,n}=4\widehat s_{a,n}\deltan^{0.47}.
\]
The scale uses all $n$ observed increments, and the same rule is applied to the response. Local covariance diagonals are floored at $10^{-12}$ before standardization. PR-SISIS uses $m_n=2$ and $q_n=d$, with $d=10$ except in the HC-R comparison with $d=20$. Ties are resolved by increasing coordinate index.

\subsection{Conditional-correlation benchmark}
\label{subsec:supp-benchmark-dgp}

The benchmark in Figure~\ref{fig:dynamic-screening-windows} retains the stochastic volatility, intraday scale, leverage, and jump specification of Section~\ref{subsec:supp-simulation-dgp}. It sets $g(t)=1$. For $u=(t-\tau)_+$, define
\[
f_T(t)=1-w_{(T-10)/390,(T+10)/390}(u),\qquad T>10,
\]
where
\[
w_{a,b}(u)=\begin{cases}0,&u\le a,\\3z^2-2z^3,&a<u<b,\quad z=(u-a)/(b-a),\\1,&u\ge b.\end{cases}
\]
For $n=390$, we use $(a,\rho_0)=(0.60,0.60)$ and $T\in\{85,100,115\}$. For $n=1560$, we use $\rho_0=0.65$, $a=0.40(1-\rho_0^2)/\rho_0$, and $T\in\{75,90,105\}$. Transition centers and widths are measured in minutes at both sampling frequencies. The second block changes at $T-30$, with the same 20-minute transition width. In the stable setting, both profiles remain equal to one. Set $\rho_t=\rho_0\{2f_T(t)-1\}$ and $c_t=0.90\{1-2f_{T-30}(t)\}$. The nontrivial correlation blocks are
\[
R_{\{1,2\},\{1,2\},t}=\begin{pmatrix}1&\rho_t\\\rho_t&1\end{pmatrix},\qquad R_{\{3,4,6\},\{3,4,6\},t}=\begin{pmatrix}1&-0.20&\sqrt{0.60}\,c_t\\-0.20&1&-\sqrt{0.60}\,c_t\\\sqrt{0.60}\,c_t&-\sqrt{0.60}\,c_t&1\end{pmatrix}.
\]
Coordinate 5 is independent of these blocks, coordinates $7,\ldots,p$ form an AR(1) block with parameter $0.10$, and all other cross-block correlations are zero. The second block can be generated from unit-variance innovations $Z_3,Z_4$ with correlation $-0.20$ by setting $Z_6=c_t(Z_3-Z_4)/\sqrt{2.40}+\sqrt{1-c_t^2}\,e_6$, where $e_6$ is an independent standard normal variable. Since the coefficients of coordinates 3 and 4 are equal, this direction has zero covariance with the response.

We use
\[
\beta_t=\left(\frac{a}{1-\rho_t^2},-\frac{a\rho_t}{1-\rho_t^2},0.40,0.40,0.40,0,\ldots,0\right)^\top,\qquad c_\varepsilon(t)=0.584-\frac{a^2}{1-\rho_t^2}.
\]
Then $R_t\beta_t=(a,0,0.32,0.32,0.40,0,\ldots,0)^\top$ and $c_{YY,t}=s(t)^2v_t$. At $n=1560$, the choice $a=0.40(1-\rho_0^2)/\rho_0$ gives target coefficients $(0.40/\rho_0,-0.40,0.40,0.40,0.40)$. Both sampling configurations therefore have target beta-min $0.40$, while their conditional scores and explained variances differ. The residual variance is positive throughout, and the smallest eigenvalue of $R_t$ is bounded below by $0.1090$. All continuous increments and inherited jumps use the same latent midpoint coefficients. Event intensities and jump amplitudes follow the reference specification, including the denominator 390 at both sampling frequencies.

The $n=1560$ specification was selected from four candidates using 100 pilot paths per setting. Each of the twelve configurations uses 500 independent evaluation paths, shared across the window rules. Pilot paths are excluded from the reported percentages, and the candidate specifications and pilot results are included in the replication materials.

The generator uses five latent substeps per observed increment. Coordinates are independently permuted on each path, with the same permutation, observations, and full-day truncation thresholds for all rules. Screening uses only the observed increments, target time, and candidate grid.

\FloatBarrier
\subsection{Additional designs}
\label{subsec:supp-dynamic-dgp}

The additional designs examine correlation reversals, changing loadings, and milder correlation changes. They retain the volatility, intraday scale, leverage, and jump specification of Section~\ref{subsec:supp-simulation-dgp}. In the correlation-reversal and changing-loading designs, set $u=(t-\tau)_+$ and $(a,b)=((T-10)/390,(T+10)/390)$ for $T\in\{45,75,105\}$. Each transition lasts 20 minutes at either sampling frequency. The stable setting holds the transition term $w_{a,b}(u)$ at zero. These designs take $g(t)=1$ and generate the response from
\[
dY_t=\beta_t^\top dX_t+s(t)\sqrt{v_t c_\varepsilon(t)}\,dW_t^\varepsilon+dJ_t^\varepsilon.
\]

In the conditional-correlation design, $R_t$ is block diagonal with $2\times2$ blocks on coordinates $(1,2)$, $(3,4)$, and $(5,6)$ having correlations $\rho_t=0.60\{1-2w_{a,b}(u)\}$, $0.25$, and $0$, respectively. Coordinates $7,\ldots,p$ form an AR(1) block with parameter $0.10$, and all cross-block correlations are zero. We set
\[
r_\tau=(0.45,0,0.50,0.50,0.40,0,\ldots,0)^\top,\qquad \beta_t=R_t^{-1}r_\tau,
\]
\[
c_\varepsilon(t)=1-r_\tau^\top R_t^{-1}r_\tau=0.44-\frac{0.2025}{1-\rho_t^2}.
\]
The residual variance lies in $[0.12359375,0.2375]$, and $c_{YY,t}=s(t)^2v_t$ with $r_t=r_\tau$. The smallest eigenvalue of $R_t$ is at least $0.40$. We use $(p,n)=(100,390),(500,390)$, and $(500,1560)$.

In the changing-loading design, $R$ has first three $2\times2$ blocks with correlations $0.25$, $0.25$, and $0.60$, with an AR(1) block of parameter $0.10$ on coordinates $7,\ldots,p$ and zero cross-block correlations. We set
\[
\beta_t=(0.25,0.25,0.25,0.25,0.65\{1-2w_{a,b}(u)\},0,\ldots,0)^\top,\qquad c_\varepsilon(t)=0.6875-\beta_{5,t}^2.
\]
Then $c_\varepsilon(t)\in[0.265,0.6875]$, $c_{YY,t}=s(t)^2v_t$, and the response correlation vector is
\[
r_t=(0.3125,0.3125,0.3125,0.3125,\beta_{5,t},0.60\beta_{5,t},0,\ldots,0)^\top.
\]
This design uses $(p,n)=(100,390)$. A milder conditional-correlation change replaces the first-pair path by $\rho_t=0.60\{1-w_{a,b}(u)\}$, which decreases from $0.60$ to zero, at the same three transition times with $(p,n)=(100,390)$.

These 19 settings use $\tau=1/2$, $C_L=0.40$, and $x=0$. Candidate durations are 32, 40, 50, 63, 79, 99, 124, and 155 minutes, with $k=nH/390$ increments for a duration of $H$ minutes. The bipower scale uses all $n$ increments. Each setting has 200 independent reference paths for selecting the fixed estimation window and 500 evaluation paths for measuring correlation error.

The continuous augmented covariance is
\[
\Sigma_t=s(t)^2v_t\begin{pmatrix}R_t&R_t\beta_t\\\beta_t^\top R_t&1\end{pmatrix}.
\]
Continuous response increments and inherited covariate jumps use the same latent midpoint coefficients. Residual increments and jumps use $c_\varepsilon(t)$, and jump amplitudes retain the denominator 390 at both sampling frequencies. Event intensities remain fixed in calendar time. The standardized covariate innovations have covariance $R_t$ and covariance vector $(-0.30,0,\ldots,0)^\top$ with the volatility innovation. Covariance integrals use the same midpoints, pre-step volatility, and latent step lengths as the increments. Coordinates are independently permuted on each path, with the same permutation used by all rules.

\FloatBarrier
\section{Additional simulations}
\label{sec:supp-additional-simulations}

\subsection{Window selection and correlation estimation}
\label{subsec:supp-benchmark-results}
\label{subsec:supp-dynamic-estimation}

 In the main benchmark, the first update selects long windows, with mean durations between 149.5 and 155.0 minutes. By the fifth update, the Early, Middle, and Late means are 82.6, 97.5, and 110.3 minutes for $(p,n)=(100,390)$, 90.6, 103.3, and 119.1 for $(500,390)$, and 50.9, 64.5, and 79.1 for $(500,1560)$. The corresponding Stable means are 153.3, 154.6, and 127.6 minutes.

 For the additional designs in Section~\ref{subsec:supp-dynamic-dgp}, we measure correlation error by
\[
\|\widehat r(k)-r_\tau\|_\infty\vee\|\widehat R(k)-R_\tau\|_{\max}.
\]
 Table~\ref{tab:supp-dynamic-estimation} compares the 155-minute window, full-matrix Lepski, and a fixed window selected by minimizing mean error on 200 independent reference paths. Errors are evaluated on 500 further paths. Full-matrix Lepski uses $\widehat k(\{1,\ldots,p_n\})$, rather than the active-column comparisons in PR-SISIS. It reduces error under pronounced changes, whereas stable or mild-change settings need not benefit.

\begin{table}[tbp]
\centering
\caption{Correlation estimation under changing dependence}
\label{tab:supp-dynamic-estimation}
\small
\setlength{\tabcolsep}{4pt}
\begin{tabular}{@{}lrrrr@{}}
\toprule
Setting & Fixed 155 (SE) & Full Lepski (SE) & Ref. $H$ & Ref. error\\
\midrule
\multicolumn{5}{@{}l}{\textit{Conditional correlations: $p=100$, $n=390$}}\\
Stable & 0.304 (0.0011) & 0.318 (0.0015) & 155 & 0.304 \\
Early (45) & 0.804 (0.0040) & 0.525 (0.0043) & 63 & 0.473 \\
Middle (75) & 0.583 (0.0043) & 0.402 (0.0028) & 99 & 0.386 \\
Late (105) & 0.388 (0.0034) & 0.343 (0.0017) & 124 & 0.340 \\
\addlinespace[3pt]
\multicolumn{5}{@{}l}{\textit{Conditional correlations: $p=500$, $n=390$}}\\
Stable & 0.358 (0.0009) & 0.371 (0.0014) & 155 & 0.358 \\
Early (45) & 0.802 (0.0040) & 0.587 (0.0044) & 79 & 0.526 \\
Middle (75) & 0.582 (0.0040) & 0.451 (0.0027) & 99 & 0.445 \\
Late (105) & 0.400 (0.0026) & 0.391 (0.0017) & 155 & 0.400 \\
\addlinespace[3pt]
\multicolumn{5}{@{}l}{\textit{Conditional correlations: $p=500$, $n=1560$}}\\
Stable & 0.183 (0.0005) & 0.251 (0.0010) & 155 & 0.183 \\
Early (45) & 0.806 (0.0021) & 0.311 (0.0013) & 50 & 0.317 \\
Middle (75) & 0.583 (0.0021) & 0.257 (0.0008) & 79 & 0.255 \\
Late (105) & 0.371 (0.0020) & 0.251 (0.0009) & 124 & 0.209 \\
\addlinespace[3pt]
\multicolumn{5}{@{}l}{\textit{Changing loading: $p=100$, $n=390$}}\\
Stable & 0.304 (0.0011) & 0.315 (0.0014) & 155 & 0.304 \\
Early (45) & 0.865 (0.0040) & 0.529 (0.0041) & 63 & 0.476 \\
Middle (75) & 0.630 (0.0039) & 0.408 (0.0027) & 99 & 0.388 \\
Late (105) & 0.415 (0.0036) & 0.345 (0.0018) & 124 & 0.341 \\
\addlinespace[3pt]
\multicolumn{5}{@{}l}{\textit{Milder correlation change: $p=100$, $n=390$}}\\
Early (45) & 0.405 (0.0032) & 0.396 (0.0030) & 124 & 0.383 \\
Middle (75) & 0.330 (0.0021) & 0.334 (0.0018) & 155 & 0.330 \\
Late (105) & 0.308 (0.0012) & 0.320 (0.0015) & 155 & 0.308 \\
\addlinespace[3pt]
\bottomrule
\end{tabular}
\smallskip
\begin{minipage}{\textwidth}\footnotesize\raggedright
\textit{Note.} Entries are mean joint maximum-norm errors on 500 evaluation paths. Parentheses give Monte Carlo standard errors. Full Lepski compares all correlation columns. Ref. $H$ minimizes mean error on 200 independent reference paths, and Ref. error evaluates that fixed duration on the evaluation paths. Durations are in minutes.
\end{minipage}
\end{table}

\subsection{Truncation and tuning}
\label{sec:supp-jump-truncation}
\label{sec:supp-tuning}

 Table~\ref{tab:supp-jump-truncation} reports two HC-R experiments. Panel A compares truncated and untruncated PR-SISIS after multiplying all three jump intensities by 32. The coordinate-specific intensity is the total across covariates, and each common event also produces the volatility jump in Section~\ref{subsec:supp-simulation-dgp}. Panel B varies the comparison constant.

\begin{table}[tbp]
\centering
\caption{Truncation and sensitivity to the comparison constant in HC-R}
\label{tab:supp-jump-truncation}
\label{tab:supp-tuning}
\small
\setlength{\tabcolsep}{5pt}
\begin{tabular}{@{}llrr@{}}
\toprule
\multicolumn{4}{@{}l}{\textit{Panel A: Jump intensities multiplied by 32}}\\
\midrule
Paths & Truncation & Replications & Sure (\%)\\
\midrule
All & Yes & 500 & 96.0\\
All & No & 500 & 77.2\\
Jump in largest window & Yes & 304 & 95.4\\
Jump in largest window & No & 304 & 64.5\\
\bottomrule
\end{tabular}

\smallskip
\begin{tabular}{@{}rrr@{}}
\toprule
\multicolumn{3}{@{}l}{\textit{Panel B: Comparison constant}}\\
\midrule
$C_L$ & M-SIS sure (\%) & PR-SISIS sure (\%)\\
\midrule
0.25 & 1.0 & 33.5\\
0.40 & 10.0 & 98.0\\
0.60 & 10.0 & 98.0\\
0.80 & 10.0 & 98.0\\
\bottomrule
\end{tabular}

\smallskip
\begin{minipage}{\textwidth}\footnotesize\raggedright
\textit{Note.} Panel A uses 500 paired paths and $C_L=0.40$. A subset of 304 paths contains a common, coordinate-specific, or residual jump in the 155-increment forward window. Panel B is a separate experiment with 200 paths per constant and common random numbers across constants. All lists contain ten coordinates. Other settings are fixed within each experiment.
\end{minipage}

\end{table}

 The five truncation settings use multipliers 3, 4, and 5 at exponent 0.47, and exponents 0.45 and 0.49 at multiplier 4. Across these settings, PR-SISIS has sure-screening percentages of 97.5\% in HC-R and 100.0\% in LW. M-SIS gives 5.5\% to 6.0\% and 99.0\%, respectively. Each setting uses 200 paths, with common random numbers within the experiment. Truncated screening also changes little between no jumps, baseline intensities, and tripled intensities.

\FloatBarrier

\clearpage
\endgroup

% Restore the IMS page geometry before its end-of-document check
\restoregeometry

\end{document}